\documentclass[a4paper,10pt]{scrartcl}
\usepackage[utf8x]{inputenc}
\usepackage[T1]{fontenc}
\usepackage[ngerman,english]{babel}
\usepackage{amsmath, amsthm, amssymb} 
\usepackage{mathtools}
\usepackage[usenames,dvipsnames]{color}
\usepackage{tikz}
\usepackage{fullpage}
\allowdisplaybreaks

\newcommand{\eps}{\varepsilon}
\newcommand{\N}{\mathbb{N}}
\newcommand{\R}{\mathbb{R}}
\newcommand{\Rn}{\R^n}

\newcommand{\nat}{\in\N}

\newcommand{\ny}{\nu}
\newcommand{\my}{\mu}
\newcommand{\phii}{\varphi}

\newcommand{\al}{\alpha}

\newcommand{\col}{\colon}
\newcommand{\Om}{\Omega}
\newcommand{\na}{\nabla}

\newcommand{\Laplace}{\Delta}
\newcommand{\Lap}{\Laplace}

\newcommand{\wbar}{\overline{w}}

\newcommand{\utilde}{\widetilde{u}}
\newcommand{\vtilde}{\widetilde{v}}

\newcommand{\Ctilde}{\widetilde{C}}

\newcommand{\kappatilde}{\widetilde{\kappa}}

\newcommand{\kappahat}{\widehat{\kappa}}

\newcommand{\ubar}{\overline{u}}

\newcommand{\ybar}{\overline{y}}

\newcommand{\rand}{\del\Omega}
\newcommand{\Ombar}{\overline{\Om}}
\newcommand{\dom}{\del \Om}
\newcommand{\amrand}{|_{\rand}}

\newcommand{\del}{\partial}

\newcommand{\delt}{\del_t}

\newcommand{\delny}{\partial_\ny}
\newcommand{\intom}{\int_\Om}
\newcommand{\intdom}{\int_{\dom}}

\newcommand{\intnT}{\int_0^T}
\newcommand{\intnt}{\int_0^t}
\newcommand{\intntau}{\int_0^\tau}
\newcommand{\intninf}{\int_0^\infty}

\newcommand{\ds}{\mbox{ds}}

\newcommand{\Tmax}{T_{max}}

\newcommand{\Liom}{L^\infty(\Om)}
\newcommand{\Lqom}{L^q(\Om)}
\newcommand{\Lpom}{L^p(\Om)}
\newcommand{\Ldom}{L^3(\Om)}
\newcommand{\Lzom}{L^2(\Om)}

\newcommand{\Wezom}{W^{1,2}(\Om)}

\newcommand{\ntom}{\Om\times(0,T)}

\newcommand{\Combar}{C(\Ombar)}

\newcommand{\Cninf}{C_0^\infty}

\newcommand{\norm}[2][]{\left\|#2\right\|_{#1}}

\newcommand{\weakto}{\rightharpoonup}

\newcommand{\upto}{\nearrow}
\newcommand{\downto}{\searrow}
\newcommand{\embeddedinto}{\hookrightarrow}

\newcommand{\limsuptTmax}{\limsup_{t\upto \Tmax}}

\newcommand{\set}[1]{\{#1\}}
\newcommand{\setl}[1]{\left\{#1\right\}}
\newcommand{\sub}{\subset}

\newtheorem{theorem}{Theorem}
\numberwithin{theorem}{section}
\newtheorem{lemma}[theorem]{Lemma}
\newtheorem{corollary}[theorem]{Corollary}
\newtheorem{proposition}[theorem]{Proposition}
\newtheorem{remark}[theorem]{Remark}
\newtheorem{definition}[theorem]{Definition}

\newtheorem{prop}[theorem]{Proposition}

\newcommand{\ue}{u_\eps}
\newcommand{\ve}{v_\eps}
\newcommand{\ye}{y_\eps}
\newcommand{\une}{u_{0,\eps}}
\newcommand{\vne}{v_{0,\eps}}
\newcommand{\uet}{u_{\eps t}}
\newcommand{\vet}{v_{\eps t}}
\newcommand{\uej}{u_{\eps_j}}
\newcommand{\vej}{v_{\eps_j}}
\renewcommand{\delt}{\frac{d}{dt}}
\newcommand{\Ceins}{A}

\newcommand{\zv}{\theta}
\DeclareMathOperator{\diam}{diam}

\author{Johannes Lankeit\thanks{Institut f\"ur Mathematik, Universit\"at Paderborn, Warburger Str. 100, 33098 Paderborn, Germany; email: \mbox{johannes.lankeit@math.upb.de}}}
\title{Eventual smoothness and asymptotics in a three-dimensional chemotaxis system with logistic source}

\begin{document}
\maketitle
\begin{abstract}
{\bf Abstract.}
 We prove existence of global weak solutions to the chemotaxis system 
\begin{align*}
 u_t&=\Lap u - \na\cdot (u\na v) +\kappa u -\my u^2\\
 v_t&=\Lap v-v+u\nonumber
\end{align*}
 under homogeneous Neumann boundary conditions in a smooth bounded convex domain $\Om\sub \Rn$, for arbitrarily small values of $\my>0$.\\
 Additionally, we show that in the three-dimensional setting, after some time, these solutions become classical solutions, provided that $\kappa$ is not too large. In this case, we also consider their large-time behaviour: We prove decay if $\kappa\leq 0$ and the existence of an absorbing set if $\kappa>0$ is sufficiently small.\\
{\bf Keywords:} chemotaxis, logistic source, existence, weak solutions, eventual smoothness \\
{\bf MSC:} 35K55 (primary), 35B65, 35Q92, 92C17, 35B40
\end{abstract}

\section{Introduction}
Starting from the pioneering work of Keller and Segel \cite{keller_segel_70}, an extensive mathematical literature has grown on the Keller-Segel model and its variants, mathematical models describing chemotaxis, that is the tendency of (micro-)organisms to adapt the direction of their (otherwise random) movement to the concentration of a signalling substance. For a survey see \cite{hillen_painter_09} or \cite{horstmann_03,horstmann_04}.

If biological phenomena where chemotaxis plays a role are modelled on not only small time scales, often growth of the population, whose density we will denote by $u$, must be taken into account.
A prototypical choice to accomplish this is the addition of logistic growth terms $+\kappa u-\my u^2$ in the evolution equation for $u$. 
Unfortunately, it is unclear whether global classical solutions to the chemotaxis-system
\begin{align}
\label{eq:limprob}
 u_t&=\Lap u - \na\cdot (u\na v) +\kappa u -\my u^2\\
 v_t&=\Lap v-v+u\nonumber\\
 \delny u\amrand&=\delny v\amrand=0\nonumber\\
 u(\cdot,0)&=u_0,\quad v(\cdot,0)=v_0,\nonumber 
\end{align}
where $\kappa\in\R$ and $u_0$, $v_0$ are given functions, exist in the smooth, bounded domain $\Om\sub\R^n$ if $n\geq 3$ and $\my>0$ is small.\\
The parabolic-elliptic simplification (where $v_t$ is replaced by $0$) of \eqref{eq:limprob} has been considered in 
\cite{tello_winkler_07}, where -- besides some study of asymptotic behaviour -- it is shown that weak solutions exist for arbitrary $\my>0$ and that they are smooth and globally classical if $\my>\frac{n-2}{n}$. In \cite{winkler_08} the existence of (very) weak solutions is proven under more general conditions. 
Under additional assupmtions, also the existence of a bounded absorbing set in $L^\infty(\Om)$ is shown.

Turning to the parabolic-parabolic system, 
important findings are given in \cite{winkler_10_boundedness}, which assert existence and uniqueness of global, smooth, bounded solutions to \eqref{eq:limprob} under the condition that $\my$ be large enough.

Additional results on existence of global solutions or even of an exponential attractor have been given in the two-dimensional case (see e.g. \cite{osaki_tsujikawa_yagi_mimura_02,osaki_yagi_02}). In this case, global solutions exist for arbitrary $\my>0$.
 
But not only the restriction to dimension $2$, 
also the inclusion of some kind of saturation effect in the chemotactic sensitivity \cite{baghaei_hesaaraki_13}, sublinear dependence of the chemotactic sensitivity on $u$ \cite{xinru_14} or even changing the second equation into one that models the consumption of the chemoattractant 
(as done in \cite{tao_11, tao_winkler_12} for $\kappa=\my=0$)
can make it possible to derive the global existence of solutions. The same can be accomplished by replacement of the secretion term $+u$ in the second equation of \eqref{eq:limprob} by $+\frac u{(1+u)^{1-\beta}}$ with some $0<\beta<\frac9{10}$, which enables the authors of \cite{nakaguchi_osaki_13} to show the existence of attractors in the corresponding dynamical system.

On the other hand, the model 
\begin{align}
\label{eq:exceed}
 u_t&=\eps\Lap u - \na\cdot (u\na v) +\kappa u -\my u^2\\
 0&=\Lap v-v+u\nonumber
\end{align}
has recently been shown to exhibit the following property \cite{lankeit_14}: 
If $\my\in(0,1)$ and the (radially symmetric) initial datum $u_0$ is large in a certain $L^p(\Om)-$space, there exists some finite time such that up to this time  any given threshold will be surpassed by solutions to \eqref{eq:exceed} for sufficiently small $\eps>0$. 
Although this demeanour may be interesting from an emergence-of-pattrens point-of-view and although solutions become very large, it still is not the same as blow-up and, in fact, also occurs in case of bounded solutions, even in space-dimension $1$ \cite{winkler_14_ctexceed}.

In \cite{winkler_11_blowupdespiteloggrowth} it is shown that in another related model, 
\begin{align*}
 u_t=\Lap u-\na\cdot(u\na v)+\kappa u- \my u^\alpha,\qquad
 v_t=\Lap v-m(t) +u,\qquad\qquad
m(t)=\frac1{|\Om|}\intom u,
\end{align*}
blow-up may occur for space-dimension $n\geq 5$ and $1<\alpha<\frac32+\frac1{2n-2}$.

Consequently, the supposition 
that any superlinear growth restriction already signifies the existence of a global, bounded solution 
does not stand unchallenged; and the question whether the above-mentioned results on the presence of global smooth solutions in similar situations 
find their analogue in the case of \eqref{eq:limprob}, the most prototypical chemotaxis system including logistic growth, 
is not clear at all.

In the present article, we therefore investigate the existence of solutions to \eqref{eq:limprob}.
More precisely, we will construct weak solutions in the sense of Definition \ref{def:soln} below.
We shall show that, in dimension $3$ and under a smallness condition on $\kappa$, they become smooth 
after some time, which also excludes finite-time blow-up from then on.
Note that this, however, does not provide any information on a small timescale.

To the aim sketched above we will then consider the approximate system 
\begin{align}
\label{eq:epsprob}
 \uet&=\Lap \ue - \na\cdot (\ue\na \ve) +\kappa \ue -\my \ue^2-\eps \ue^\zv\\
 \vet&=\Lap \ve-\ve+\ue,\nonumber
\end{align}
for $\zv>n+2$ with nonnegative initial values $u_{0,\eps}\in \Combar$ and $v_{0,\eps}\in W^{1,n+1}(\Om)$, 
 where global classical solutions are quickly seen to exist, and derive estimates finally allowing for compactness arguments, which will provide the existence of a weak solution to \eqref{eq:limprob} in Proposition \ref{prop:conv} and Lemma \ref{lem:limesissoln}.

We will employ the estimates from Section \ref{sec:small} 
to conclude that a solution must become small in an appropriate sense after some time. 
This, in turn, will be the starting point for an ODE comparison argument for the quantity $\intom \ue^2(t)+\intom |\na \ve(t)|^4$, whose thereby-obtained 
boundedness in conjunction with estimates on the Neumann heat semigroup results in eventual boundedness and hence in eventual smoothness of $(u,v)$. 
We finally arrive at the following result:
\begin{theorem}
\label{thm:eins}
Let $\Om\sub\Rn$ be a smooth bounded convex domain and $u_0\in \Lzom$, $v_0\in \Wezom$ nonnegative. 
Let $\kappa\in\R, \my>0$.

Then there is a nonnegative weak solution $(u,v)$ (in the sense of Definition \ref{def:soln} below) to \eqref{eq:limprob} with initial data $u_0$, $v_0$.\\
It can be approximated in the sense of $a.e.$-convergence by solutions of \eqref{eq:epsprob}.  
\\
 Furthermore, if $n=3$, for any $\my>0$ there exists $\kappa_0>0$ such that if $\kappa<\kappa_0$, there is $T>0$ such that $u$ and $v$ are a classical solution of $\eqref{eq:limprob}$ for $t>T$.\\
Moreover, in this case, there are $C>0$ and $\al>0$ such that for any $t>T$
\[
 \norm[{C^{2+\al,1+\frac\al2}(\Ombar\times [t,t+1])}]{u} +\norm[{C^{2+\al,1+\frac\al2}(\Ombar\times [t,t+1])}]{v}\leq C.
\]
\end{theorem}

\begin{remark}
 Because we have adopted a weak concept of solution, it is conceivable that solutions to \eqref{eq:limprob} are not unique. 
Investigation of this issue is beyond the scope of the present work and we state the following theorems only for solutions as provided by Theorem \ref{thm:eins}.
\end{remark}

Besides the aforementioned results about attractors, little is known about asymptotic behaviour of solutions to models like \eqref{eq:limprob}. 
Recently, in \cite{winkler_14_global_asymptotic} convergence to the positive homogeneous equilibrium was found for values of $\my$ being sufficiently large as compared to the chemotactic sensitivity. 

The richness of dynamics and pattern formation exhibited by chemotaxis models with growth \cite{painter_hillen_11,kuto_osaki_sakurai_tsujikawa_12} however indicates that any speculation about 
asymptotical behaviour, especially about convergence to homogeneous states, should be backed by rigorous examinations. 

In the situation of \eqref{eq:limprob}, we can summarize the long-term behaviour as follows: 
If $\kappa\leq0$, solutions will converge to the trivial steady state - and any formation of interesting patterns has to take place on intermediary timescales.
\begin{theorem}
 \label{thm:kappazero}
 Let $\my>0$, $\kappa\leq0$. Let $\Om\sub\R^3$ be a smooth bounded convex domain and let $(u,v)$ be the solution to \eqref{eq:limprob} provided by Theorem \ref{thm:eins}. 
 Then 
\[
 (u(t),v(t))\to (0,0) \qquad \text{as } t\to \infty
\]
 in the sense of uniform convergence.
\end{theorem}
\begin{remark}
 The same convergence result can be given for any classical solution of \eqref{eq:limprob} for $\my>0$, $\kappa\leq 0$ in $\Om\sub\R^3$ as above. In this case, only minor adaptions of the proofs become necessary. 
\end{remark}

If $\kappa$ is positive and sufficiently small, we can assert the existence of an absorbing set in the following sense:
\begin{theorem}
 \label{thm:kappapositive}
 Let $\Om\sub\R^3$ be a smooth bounded convex domain. Then for any $\my>0$ there is $\kappa_0>0$ such that for all $\kappa\in(0,\kappa_0)$, there is $\al>0$ and a bounded set $B_{\my,\kappa}\sub (C^{2+\al}(\Ombar))^2$ such that for all $(u_0,v_0)\in\Lzom\times \Wezom$, the corresponding solution $(u,v)$ as constructed in Theorem \ref{thm:eins} admits the 
existence of $T>0$ such that
\[
 (u(t),v(t))\in B_{\my,\kappa} \qquad \text{for all } t>T.
\]
 Furthermore, for each fixed $\my>0$, 
\[
 \diam_{\Liom\times W^{1,\infty}(\Om)}(B_{\my,\kappa})\to 0 \qquad \text{as }\kappa\downto 0.
\]
\end{theorem}

Further steps in this direction may hopefully lead to an even more detailled insight, much in the spirit of \cite{nakaguchi_efendiev_08,aida_tsujikawa_efendiev_yagi_mimura_06}, into the long-time behaviour of solutions to \eqref{eq:limprob} in dimension $3$ for small, positive $\my$. 

\begin{remark}
 In the calculations below, we will assume that $\my>0$ is a fixed number.


 Throughout the article, we fix $\Om\sub \Rn$ to be a convex bounded domain with smooth boundary and $u_0\in \Lzom$, $v_0\in \Wezom$ nonnegative. 
 Also, let $\zv$ denote a number satisfying $\zv>n+2$.
\end{remark}

\section{Existence of approximate solutions}
The system \eqref{eq:epsprob} has a unique, global, classical solution. 
At a first glance, the source term $f(s)=\kappa s-\my s^2 -\eps s^\zv$ seems to satisfy the condition $f(s)\leq a-\my_0 s^2$ from Theorem 0.1 of \cite{winkler_10_boundedness}, which would provide a global solution, but as $\my_0$ depends on $a$, this theorem is not applicable in the present case. 
Even tracing the dependece of $\my_0$ on $a$ does not improve the situation.

We therefore use Lemma 1.1 of the same article, 
 which asserts the local existence of a unique classical solution $(\ue,\ve)$ to \eqref{eq:epsprob} for initial data $u_{0\eps}\in \Combar$, $v_{0\eps}\in W^{1,n+1}(\Om)$. More specifically, it implies that this solution exists on a time interval $[0,\Tmax)$, $\Tmax\in(0,\infty]$, and satisfies 
\[
 \limsuptTmax \left(\norm[\Liom]{\ue(t)}+\norm[W^{1,\infty}(\Om)]{\ve(t)}\right) = \infty
\]
if $\Tmax<\infty$. 
Hence, in order to show the global existence of this solution,
it is sufficient to derive boundedness of $\ue, \ve$ and $\na \ve$.

Our means of pursuing this aim will be 

\begin{proposition}
\label{prop:epsbounded}
 Let $q>n+2$. 
 Let $(u,v)$ be a nonnegative classical solution of  
\begin{align*}
 u_t=&\Lap u-\na\cdot(u\na v)+f(u),\\
 v_t=&\Lap v-v+u
\end{align*}
in $\Om\times [0,T]$, $T>0$, with homogeneous Neumann boundary conditions, for initial data $v_0\in W^{1,\infty}(\Om)$, $u_0\in\Liom$ and some function $f$ satisfying $f(s)\leq C_0$ for all $s>0$ with some $C_0>0$. Furthermore, assume that there exists $C>0$ such that $u$ satisfies
\[
 \left(\intnT\intom u^q\right)^{\frac1q}\leq C.
\]
Then $u$, $v$ and $\na v$ are bounded in $\Om\times[0,T]$.
\end{proposition}
\begin{proof}
Denote by $C_1, C_2, C_3$ the constants provided by Lemma 1.3 of \cite{winkler_10_aggregationvs} such that 
\begin{equation}
 \label{eq:prop1defc1}
 \norm[\Liom]{\na e^{\tau\Lap}w}\leq C_1\norm[\Liom]{\na w}
\end{equation}
for all $w\in W^{1,\infty}(\Om)$ and
\begin{equation}
 \label{eq:prop1defc2}
 \norm[\Liom]{\na e^{\tau\Lap}w}\leq C_2 ( 1+ \tau^{-\frac12-\frac{n}{2q}})\norm[\Lqom]{w}
\end{equation}
for all $w\in\Lqom$ 
as well as 
\begin{equation}
\label{eq:prop1defc3}
 \norm[\Liom]{e^{\tau\Lap}\na\cdot w}\leq C_3 ( 1+ \tau^{-\frac12-\frac{n}{2q}})\norm[\Lqom]{w}
\end{equation}
for $w\in L^q(\Om,\Rn)$. 
Here, $e^{\tau\Lap}\na\cdot$ denotes the extension of the corresponding operator on $(C_0^\infty(\Om))^n$ to a continuous operator from $L^q(\Om,\Rn)$ to $\Liom$, see \cite[Lemma 1.3]{winkler_10_aggregationvs}. 
Since $(-\frac12-\frac{n}{2q})\cdot\frac{q}{q-1} = -\frac12 \frac{q+n}{q} \frac{q}{q-1}=-\frac12(1+\frac{n+1}{q-1})>-\frac12(1+\frac{n+1}{(n+2)-1})=-1$, 
\begin{equation}
\label{eq:prop1defc4}
 C_4=\left(\intnT (1+(T-s)^{-\frac12-\frac{n}{2q}})^{\frac{q}{q-1}} \ds\right)^{\frac{q-1}{q}}
\end{equation}
is finite.

 Let $t\in[0,T]$. 
Employing \eqref{eq:prop1defc1} and \eqref{eq:prop1defc2} in the variations-of-constants formula for $v$, we obtain 
\begin{align}
\label{eq:prop1navbd}
 \norm[\Liom]{\na v(t)}\leq& \norm[\Liom]{\na e^{t(\Lap-1)} v_0}+\intnt\norm[\Liom]{\na e^{(t-s)(\Lap-1)}u(s) }\ds\notag\\
  \leq& C_1\norm[\Liom]{\na v_0} + C_2\intnt (1+(t-s)^{-\frac12-\frac{n}{2q}})\norm[\Lqom]{u(s)}\ds \notag\\
  \leq& C_1\norm[\Liom]{\na v_0} + C_2\left(\intnt (1+(t-s)^{-\frac12-\frac{n}{2q}})^{\frac{q}{q-1}} \ds\right)^{\frac{q-1}q} \left(\intnt\norm[\Lqom]{u(s)}^q\ds\right)^{\frac1q}\notag\\
 \leq& C_1\norm[\Liom]{\na v_0} +C_2C_4C =: C_5.
\end{align}
We represent also $u$ in terms of the semigroup, use the order-preserving property of the heat semigroup and estimate with the help of \eqref{eq:prop1defc3} to see that 
 \begin{align*}
  0\leq u(t) =& e^{t\Lap} u_0 + \intnt e^{(t-s)\Lap}\na\cdot(u(s)\na v(s))\ds +\intnt e^{(t-s)\Lap}f(u(s))\ds \\
\leq& \norm[\Liom]{u_0}  + C_3\intnt ( 1+ (t-s)^{-\frac12-\frac{n}{2q}})\norm[\Lqom]{u(s)}\norm[\Liom]{\na v(s)}\ds+ TC_0.
\end{align*}
Another application of H\"older's inequality, now in time, in combination with \eqref{eq:prop1navbd} and \eqref{eq:prop1defc4} gives 
\begin{align*}
0\leq u(t) \le& \norm[\Liom]{u_0}  + C_3C_5 \left(\intnt( 1+ (t-s)^{-\frac12-\frac{n}{2q}})^{\frac{q}{q-1}}\right)^{\frac{q-1}q} \left(\intnt \norm[\Lqom]{u(s)}^q\right)^{\frac1q}+ TC_0\\
\leq& \norm[\Liom]{u_0} + C_3C_5C_4C+ TC_0=:C_6.
 \end{align*}
Boundedness of $v$ on $\Om\times[0,T]$ then is an easy consequence:
\begin{align*}
 0\leq v(t)\leq& \norm[\Liom]{e^{t(\Lap-1)}v_0}+\intnt\norm[\Liom]{e^{(t-s)(\Lap-1)} u(s)}\ds
\leq\norm[\Liom]{v_0}+\intnt C_6\ds
\end{align*}
for all $t\in[0,T]$.
\end{proof}

For given nonnegative $u_0\in \Lzom$, 
$v_0\in\Wezom$ and $\eps>0$, we choose $u_{0,\eps}\in \Combar$, $v_{0,\eps}\in W^{1,n+1}(\Om)$ nonnegative such that 
\begin{equation}
\label{eq:init} 
\norm[\Lzom]{u_0-u_{0,\eps}}\leq \min\set{\eps,1},\qquad\norm[\Wezom]{v_0-v_{0,\eps}}\leq \min\set{\eps,1}.
\end{equation}
From now on, by $(\ue, \ve)$ we denote the unique classical solution on $[0,\Tmax)$ to \eqref{eq:epsprob} with initial data $u_{0,\eps}$ and $v_{0,\eps}$. 
Proposition \ref{prop:epsbounded} in conjunction with the next two lemmata and Lemma 1.1 of \cite{winkler_10_boundedness} will show that, indeed, $\Tmax=\infty$.
\\Note that, by \eqref{eq:init}, in the following lemmata estimates in terms of $u_{0,\eps}$ or $v_{0,\eps}$ can be made $\eps$-independent by retreating to the corresponding integral of $u_0$ or $v_0$ plus $1$.


\section{Estimates}
\label{sec:estimates}
In this section we present 
estimates for different quantities involving $\ue$ and $\ve$ respectively, which can be obtained more or less directly from \eqref{eq:epsprob} together with ODE comparison arguments.
In the following, denote 
\[
 \kappa_+:=\max\set{\kappa,0}.
\]

\begin{lemma}
\label{lem:ule}
For any $\eps>0$, the function $\ue$ satisfies
\[
 \intom \ue(t)\leq \max\setl{\intom \une, \frac{\kappa_+|\Om|}{\my}}
\]
for $t>0$. Furthermore, 
\[
 \limsup_{t\to \infty} \intom \ue(t) \leq \frac{\kappa_+|\Om|}{\my},
\]
uniformly in $\eps>0$.
\end{lemma}
\begin{proof}
By H\"older's  inequality, $\left(\intom \ue\right)^2\leq \left(\intom \ue^2\right) |\Om|.$
Hence, integration of the first equation of \eqref{eq:epsprob} yields
\begin{align}
\label{eq:firsteq_integrated}
 \left(\intom \ue\right)_t=\intom \uet &\leq 0-0+\kappa_+\intom \ue-\my \intom \ue^2-\eps \intom \ue^\zv\\
&\leq \kappa_+\intom \ue -\frac\my{|\Om|}\left(\intom \ue\right)^2.\nonumber
\end{align}
The claim can be seen by solving the logistic ODE. 
\end{proof}

\begin{lemma}
\label{lem:ulz}
Let $\kappa>0$, let $T>0$. Then there exists $C>0$ such that for all $\eps>0$
\[
 \intnT \intom \ue^2 +\frac{\eps}{\my}\intnT\intom \ue^\zv \;\leq\;  \frac{\kappa_+}{\my} \max\setl{\intom \une, \frac{\kappa_+|\Om|}{\my}} T+\frac1{\my}\intom \une\;\leq\; C. 
\] 
\end{lemma}
\begin{proof}
 The estimate
\[
 \intnT \intom \ue^2 +\frac{\eps}{\my}\intnT\intom \ue^\zv\leq \frac{\kappa_+}{\my}\intnT\intom \ue +\frac1{\my}\intom \une-\frac1{\my}\intom \ue(T)\leq \frac{\kappa_+}{\my} \max\setl{\intom \une, \frac{\kappa_+|\Om|}{\my}} T+\frac1{\my}\intom \une
\]
 results from \eqref{eq:firsteq_integrated} after time-integration.
\end{proof}

Also for the second component of the solution some basic estimates are available:
\begin{lemma}
 \label{lem:vlevlz} Let $\kappa\in\R$, $\eps>0$. 
 The inequality
\[
 \intom \ve(t) \leq \max\setl{\intom \une, \frac{\kappa_+|\Om|}{\my}, \intom \vne}
\]
holds as well as 
\begin{align*}
 \intom \ve^2(t)+\intnt\intom \ve^2 \leq \frac{\kappa_+}{\my} \max\setl{\intom \une, \frac{\kappa_+|\Om|}{\my}} t+\frac1{\my}\intom \une+\intom \vne^2
\end{align*}
for all $t>0$.
\end{lemma}
\begin{proof}
Integrating the second equation of \eqref{eq:epsprob} gives, by Lemma \ref{lem:ule},
 \begin{align*}
 \delt \intom \ve(t)=\intom \vet(t) &= \intom \Lap \ve(t) -\intom \ve(t) +\intom \ue(t)
 \leq -\intom \ve(t) + \max\setl{\intom \une, \frac{\kappa_+|\Om|}{\my}}
\end{align*}
for $t>0$, an ODI for $\intom \ve$, whose solution directly shows
\begin{equation}\label{eq:intvdecay}
 \intom \ve(t)\leq \max\setl{\intom \une, \frac{\kappa_+|\Om|}{\my}} +e^{-t}\intom v_{0,\eps}
\end{equation}
and hence the first part of the assertion. 

As to the second part, we derive an ODI for $\frac12\intom\ve^2$ in quite the same way: 
For $t>0$, by Young's inequality 
\begin{align*}
 \frac12\delt\intom \ve^2(t)=\intom \ve(t)\vet(t)&=\intom \ve(t)\Lap \ve(t) -\intom \ve^2(t)+\intom \ue(t) \ve(t)\\
 &\leq -\intom |\na \ve(t)|^2-\intom \ve^2(t)+\frac12\intom \ue^2(t)+\frac12\intom \ve^2(t)\\
 &\leq -\frac12\intom \ve^2(t)+\frac12\intom \ue^2(t).
\end{align*}
Integrating this with respect to the time variable, so that we can use the bound from Lemma \ref{lem:ulz} on $\ue^2$, we obtain  
\begin{align*}
 \intom \ve^2(t)-\intom \vne^2&\leq -\intnt\intom \ve^2 +\intnt\intom \ue^2\\
 &\leq -\intnt\intom \ve^2 + \frac{\kappa_+}{\my} \max\setl{\intom \une, \frac{\kappa_+|\Om|}{\my}} t+\frac1\my\intom \une
\end{align*}
for any $t>0$ and the claim follows.
\end{proof}

The next lemma gives estimates on the derivatives of $v$.
\begin{lemma}
\label{lem:vderivatives}
Let $\kappa\in\R, \eps>0$. The solutions of \eqref{eq:epsprob} satisfy, for all $t>0$, 
\[
  \left[ \intom |\na \ve(t)|^2 +\frac{1}{\my} \intom \ue(t)\right] \leq \max\setl{ \int|\na \vne|^2+\frac{1}{\my}\intom \une,\frac{\kappa_++1}{\my} \max\setl{\intom \une, \frac{\kappa_+|\Om|}{\my}}}
\]
and
\[
 \intnt\intom |\Lap \ve(t)|^2\leq \frac{\kappa_+}{\my}\max\setl{\intom \une, \frac{\kappa_+|\Om|}{\my}}t+\intom |\na \vne|^2 +\frac{1}{\my} \intom \une.
\]
\end{lemma}
\begin{proof}
Integration by parts and Young's inequality result in
 \begin{align}
 \delt\left[ \intom |\na \ve|^2 +\frac{1}{\my} \intom \ue\right]
\leq&-2\intom|\Lap\ve|^2+2\intom\Lap\ve\ve-2\intom\Lap\ve\ue+\frac{\kappa_+}\my\intom\ue-
\intom\ue^2-\frac{\eps}\my\intom\ue^\zv\nonumber\\
\leq&-2\intom|\Lap\ve|^2-2\intom|\na\ve|^2+\intom|\Lap\ve|^2+\intom\ue^2-\intom\ue^2+\frac{\kappa_+}\my\intom\ue\nonumber\\
\leq&-\intom|\Lap\ve|^2-\intom|\na\ve|^2-\frac1\my\intom\ue+\frac{\kappa_++1}{\my}\intom\ue\label{eq:resort_terms_here}
\end{align}
on $(0,\infty)$. From this, we can conclude by Lemma \ref{lem:ule}
\begin{align*}
 \delt\left[ \intom |\na \ve|^2 +\frac{1}{\my} \intom \ue\right]
\leq&-\left[\intom|\na\ve|^2+\frac1\my\intom\ue\right]+\frac{\kappa_++1}{\my}\max\setl{\intom \une, \frac{\kappa_+|\Om|}{\my}}\nonumber
%
\end{align*}
on $(0,\infty)$ and hence the claim follows by comparison with the solution of $y'=-y+const$.
Re-sorting the terms in \eqref{eq:resort_terms_here} moreover gives 
\begin{align*}
 \intom |\Lap \ve(t)|^2 \leq  - \intom |\na \ve(t)|^2 + \frac{\kappa_+}{\my}\intom \ue(t) - \delt\left[ \intom |\na \ve(t)|^2 +\frac{1}{\my} \intom \ue(t)\right]
\end{align*}
for $t>0$, and therefore 
\begin{align*}
 \intnt\intom |\Lap \ve|^2 &\leq   \frac{\kappa_+}{\my}\max\setl{\intom \une, \frac{\kappa_+|\Om|}{\my}}t - \left[ \intom |\na \ve(t)|^2 +\frac{1}{\my} \intom \ue(t)\right] + \intom |\na \vne|^2 +\frac{1}{\my} \intom \une\\
&\leq \frac{\kappa_+}{\my}\max\setl{\intom \une, \frac{\kappa_+|\Om|}{\my}}t+\intom |\na \vne|^2 +\frac{1}{\my} \intom \une.\qedhere
\end{align*}
\end{proof}

The bounds that have been derived so far can be combined to yield 

\begin{lemma}
 \label{lem:equi}
 Let $\kappa\in\R$. For any $T>0$, there exists a constant $C=C(T,\my,\kappa_+,\norm[\Lzom]{u_0},\norm[\Wezom]{v_0})$ 
such that for all $\eps>0$
\[
 \frac12\intnT\intom \frac{|\na\ue|^2}{1+\ue}+\my\intnT\intom \ue^2\log(1+\ue)+\eps\intnT\intom \ue^\zv\log(1+\ue)\leq C.
\]
In particular: The families $\set{\ue^2}_{\eps\in(0,1)}$ and $\set{\eps\ue^\zv}_{\eps\in(0,1)}$ are equi-integrable over $\Om\times(0,T)$.
\end{lemma}
\begin{proof}
Let $T>0$. Testing the first equation of \eqref{eq:epsprob} with $\log(1+\ue)$ and integrating by parts gives
\begin{align*}
\intom \uet\log(1+\ue) \leq &-\intom\frac{|\na \ue|^2}{1+\ue} +\intom \frac{\ue\na \ve\cdot\na \ue}{1+\ue}+\kappa_+ \intom \ue\log(1+\ue)\\&-\my\intom \ue^2\log(1+\ue)-\eps\intom \ue^\zv\log(1+\ue),
\end{align*}
which, using $((1+\ue)\log(1+\ue)-\ue)_t=\uet\log(1+\ue)$, can be turned into 
\begin{align*}
 \my\intom \ue^2\log(1+\ue)+\eps\intom\ue^\zv\log(1+\ue)\leq&-\intom \frac{|\na \ue|^2}{1+\ue}+\intom \frac{\ue\na \ve\cdot\na \ue}{1+\ue}+\kappa_+\intom \ue\log(1+\ue)\\&-\intom [(1+\ue)\log(1+\ue)-\ue]_t.
\end{align*}
Integration in time hence yields 
\begin{align}
\label{eq:secondintegral}
I\coloneqq& \my\intnT\intom \ue^2\log(1+\ue)+\eps\intnT\intom \ue^\zv\log(1+\ue)\nonumber\\
 \leq& -\intnT\intom \frac{|\na \ue|^2}{1+\ue}+\intnT\intom \frac{\ue\na \ve\cdot\na \ue}{1+\ue} + \kappa_+ \intnT\intom \ue\log(1+\ue)\nonumber\\
&-\intom ((1+\ue(T))\log(1+\ue(T))-\ue(T))+\intom ((1+\une)\log(1+\une)-\une)\nonumber\\
\leq& -\intnT\intom\frac{|\na \ue|^2}{1+\ue}+\intnT\intom \frac{\ue\na \ve\cdot\na \ue}{1+\ue}+\kappa_+ \intnT\intom \ue\log(1+\ue)\nonumber\\
&+\intom \ue(T)+\intom (1+\une)\log(1+\une).
\end{align}

We integrate the second term by parts:
\begin{align*}
 \intnT\intom \frac{\ue\na \ue\cdot\na \ve}{1+\ue}=&-\intnT\intom \frac{\ue^2}{1+\ue}\Lap \ve - \intnT\intom \ue\na \ve\cdot\na\left(\frac \ue{1+\ue}\right)\\=&-\intnT\intom \frac{\ue^2}{1+\ue} \Lap \ve-\intnT\intom \ue\na \ve\cdot\na \ue \frac{1+\ue-\ue}{(1+\ue)^2}.
\end{align*}
Inserting this into \eqref{eq:secondintegral} then results in
\begin{align*}
 I \leq& -\intnT\intom\frac{|\na \ue|^2}{1+\ue}-\intnT\intom\frac{\ue^2}{1+\ue}\Lap \ve-\intnT\intom \frac{\ue\na \ve}{(1+\ue)^{\frac32}}\cdot\frac{\na \ue}{(1+\ue)^{\frac12}}\\&+\kappa_+\intnT\intom \ue\log(1+\ue)+\intom \ue(T)+\intom (1+\une)\log(1+\une),
\end{align*}
where application of the trivial inequality $\frac{\ue}{(1+\ue)^{\frac32}}\leq 1$ gives rise to 
\begin{align*}
I \leq& -\intnT\intom \frac{|\na \ue|^2}{1+\ue}-\intnT\intom \frac \ue{1+\ue} \ue\Lap \ve+\intnT\intom|\na \ve| \frac{|\na \ue|}{\sqrt{1+\ue}}\\ &+\kappa_+\intnT\intom \ue\log(1+\ue)+\intom \ue(T)+\intom(1+\une)\log(1+\une).
\end{align*}
Estimating $\frac{\ue}{1+\ue}\leq 1$, $\log(1+\ue)\leq \ue$ and employing Young's inequality shows
\begin{align*}
I \leq& -\intnT\intom\frac{|\na \ue|^2}{1+\ue}+\frac12\intnT\intom \ue^2+\frac12\intnT\intom |\Lap \ve|^2+\frac12\intnT\intom|\na \ve|^2 +\frac12\intnT\intom \frac{|\na \ue|^2}{1+\ue}\\&+\kappa_+\intnT\intom \ue^2+\intom \ue(T)+\intom (1+\une)\log(1+\une)\\
=&-\frac12\intnT\intom \frac{|\na \ue|^2}{1+\ue}+\left(\kappa_++\frac12\right)\intnT\intom \ue^2+\intom \ue(T)+\frac12\intnT\intom |\Lap \ve|^2+\frac12\intnT\intom |\na \ve|^2\\&+\intom (1+\une)\log(1+\une).
\end{align*}
And if we compile the bounds provided by Lemmata \ref{lem:ulz}, \ref{lem:vderivatives} and \ref{lem:ule}, we arrive at 
\begin{align*}
I+\frac12\intnT\intom \frac{|\na \ue|^2}{1+\ue} \leq& \left(\kappa_++\frac12\right)\frac1\my\left(\kappa_+ \max\setl{\intom \une, \frac{\kappa_+|\Om|}{\my}} T+\intom \une\right)+\max\setl{\intom \une, \frac{\kappa_+|\Om|}{\my}}\\
&+  \frac12\left(\frac{\kappa_+}{\my}\max\setl{\intom \une, \frac{\kappa_+|\Om|}{\my}}T+\intom |\na \vne|^2 +\frac{1}{\my} \intom \une\right)\\
&+\frac12 T \max\setl{ \int|\na \vne|^2+\frac{1}{\my}\intom \une,\frac{\kappa_++1}{\my} \max\setl{\intom \une, \frac{\kappa_+|\Om|}{\my}}} \\
&+{\intom (1+\une)\log(1+\une)=:C.}\hfill\qedhere
\end{align*}
\end{proof}

From the bound on $\intnT\intom \frac{|\na\ue|^2}{1+\ue}$ we can extract information on the behaviour of the spatial gradient of $u$.
\begin{lemma}
 \label{lem:uinlewevd}
Let $\kappa\in\R$. For all $T>0$ there is $C>0$ such that for all $\eps>0$ 
\[
 \norm[L^{\frac43}((0,T), W^{1,\frac43}(\Om))]{\ue}\leq C.
\]
\end{lemma}
\begin{proof}
 Denote by $C_1$ the constant provided by Lemma \ref{lem:ulz}
 and by $C_2$ that of Lemma \ref{lem:equi}. Then, by H\"older's and Young's inequalities,
\begin{align*}
 \norm[L^{\frac43}((0,T), W^{1,\frac43}(\Om))]{\ue}^{\frac43} =& \intnT\norm[W^{1,\frac43}(\Om)]{\ue}^{\frac43}=\intnT\left(\intom \ue^{\frac43}+\intom|\na\ue|^{\frac43}\right)\\
\leq&\left(\intnT\intom\ue^2\right)^{\frac23}(|\Om|T)^{\frac13}+\intnT\intom \frac{|\na \ue|^{\frac43}}{(1+\ue)^{\frac23}}(1+\ue)^{\frac23}\\
\le& C_1^{\frac23}(|\Om|T)^{\frac13}+\frac23\intnT\intom\frac{|\na\ue|^2}{1+\ue}+\frac13\intnT\intom(1+\ue)^2\\
\le&C_1^{\frac23}(|\Om|T)^{\frac13}+\frac43 C_2+\frac23 T |\Om| + \frac23 C_1 =: C.\qedhere
\end{align*}
\end{proof}

In order to gain convergence results from Aubin-Lions-type lemmas, we need 
some information on the time derivative. The following lemma provides this kind of information.
\begin{lemma}
 \label{lem:uetbd}
 Let $\kappa\in\R$ and $T>0$. Then there is $C>0$ such that for all $\eps>0$ 
\[
 \norm[L^1((0,T),(W^{2,\infty}(\Om))^*)]{\uet}\leq C.
\]
\end{lemma}
\begin{proof}
Definition of the norm and integration by parts in \eqref{eq:epsprob} lead us to 
 \begin{align*}
  \intnT&\sup_{\norm[W^{2,\infty}(\Om)]{\phii}\leq 1}\left|\intom \uet\phii\right|\\
&\leq\intnT \sup_{\norm[W^{2,\infty}(\Om)]{\phii}\leq 1}\left(\left|\intom \ue\Lap\phii \right|+\left|\intom \ue\na\ve\cdot\na\phii\right|+\left|\kappa\intom \ue\phii\right|+\my\left|\intom \ue^2\phii\right|+\eps\left|\intom \ue^\zv\phii\right|\right), 
 \end{align*}
 where we can use $\norm[W^{2,\infty}(\Om)]{\phii}\leq 1$ and Young's inequality to see
 \begin{align*}
  \norm[L^1((0,T),(W^{2,\infty}(\Om))^*)]{\uet}\leq \intnT\left(\intom\ue+\frac12\intom\ue^2+\frac12\intom|\na\ve|^2+|\kappa|\intom\ue+\my\intom\ue^2+\eps\intom\ue^\zv\right)
 \end{align*}
 and infer boundedness of this norm, independent of $\eps$, 
 from Lemmata \ref{lem:ule}, \ref{lem:ulz} and \ref{lem:vderivatives}.
\end{proof}

The space in which the spatial gradient is known to be bounded can be improved if a bound on $u$ is assumed.
\begin{lemma}
 \label{lem:ueinwezifbd}
 Let $\kappa\in\R$ and let $[T_1,T_2]$ be an interval such that there exists a constant $M$ satisfying 
\[
 \norm[\Liom]{\ue(t)}\leq M
\]
for all $t\in[T_1,T_2]$ and $\eps>0$. 
Then there is $C>0$ such that for all $\eps>0$ 
\[
 \norm[L^2((T_1,T_2);\Lzom)]{\na\ue}\leq C.
\]
\end{lemma}

\begin{proof}
 By Lemma \ref{lem:equi} we can find $\Ctilde>0$ such that 
\[
 \intnT\intom \frac{|\na \ue|^2}{1+\ue} \leq \Ctilde
\]
for all $\eps>0$, ergo, setting $C=(1+M)\Ctilde$, 
\[
 \int_{T_1}^{T_2}\intom |\na\ue|^2\leq \int_{T_1}^{T_2}\intom \frac{1+M}{1+\ue}|\na\ue|^2\leq(1+M)\Ctilde=C.\qedhere
\]
\end{proof}
Under similar conditions, also the time derivative is bounded in a better space.
\begin{lemma}
 \label{lem:uetlzwezdual}
 Let $\kappa\in\R$ and let $[T_1,T_2]$ be an interval such that there exists a constant $M$ satisfying
 \[
  \norm[\Liom]{\ue(t)}+\norm[W^{1,\infty}(\Om)]{\ve(t)}\leq M 
 \]
 for all $t\in[T_1,T_2]$ and $\eps>0$.
 Then there is $C>0$ such that for all  $\eps>0$
 \[
  \norm[L^2((T_1,T_2);(W^{1,2}(\Om))^*)]{\uet}\leq C.
 \]
\end{lemma}
\begin{proof}
 Let $\phii$ be an element of $L^2((T_1,T_2);(W^{1,2}(\Om))$ with norm $1$.
 
 Let $\Ctilde$ be the bound on $\norm[L^2((T_1,T_2);\Lzom)]{\na \ue}$ provided by Lemma \ref{lem:ueinwezifbd}. Then
 \begin{align*}
  \left|\int_{T_1}^{T_2}\intom \uet\phii\right| \leq& \left|\int_{T_1}^{T_2}\intom \na \ue\cdot\na\phii\right| + \left|\int_{T_1}^{T_2}\intom \ue\na \ve\cdot\na\phii\right|
 +|\int_{T_1}^{T_2}\intom (\kappa \ue - \my \ue^2 - \eps \ue^\zv) \phii |\\
 \leq& \norm[L^2((T_1,T_2);\Lzom)]{\na\ue}\norm[L^2((T_1,T_2);\Lzom)]{\na\phii} \\
&+ (\sup_{t\in[T_1,T_2]} \norm[\Liom]{\na \ve(t)}) \norm[L^2((T_1,T_2);\Lzom)]{\ue} \norm[L^2((T_1,T_2);\Lzom)]{\na\phii}\\
& +\sqrt{(T_2-T_1)|\Om|}\norm[L^\infty(\Om\times(T_1,T_2))]{\kappa\ue+\my\ue^2+\ue^\zv} \norm[L^2((T_1,T_2);\Lzom)]{\phii}\\
\leq& \Ctilde + M\sqrt{(T_2-T_1)|\Om|}M + \sqrt{(T_2-T_1)|\Om|} (|\kappa| M+\my M^2+M^\zv)=:C
 \end{align*}
 and hence boundedness of $\uet$ in $(L^2((T_1,T_2);(W^{1,2}(\Om)))^*$ follows.
\end{proof}

\section{Preservation of smallness}
\label{sec:small}
In the last two lemmata, we have seen that boundedness can provide bounds also for derivatives. It will as well be important in establishing regularization effects. 
Therefore, in this section we will derive this boundedness and to this aim proceed as follows:
At first we will prepare some estimates on $\ye(t):=\intom \ue^2(t)+\intom|\na\ve|^4$.
These will establish that $\ye$ satisfies a differential inequality with a polynomial right hand side;
we will show that this polynomial has a positive root and $\ye$ eventually undermatches its value.
Finally, we will use the bounds just gained to improve them to $L^\infty$-bounds for the solutions under consideration.

At first we state the following easy consequence of Poincar\'e's inequality.
\begin{lemma}
 \label{lem:poincare}
 If we denote $\wbar=\frac1{|\Om|}\intom w$, then 
\[
 \intom w^2\leq C_P\intom|\na w|^2+|\Om|\wbar^2, 
\]
for all $w\in\Wezom$, where $C_P$ is the Poincar\'e-constant of $\Om$, defined by $\intom (w-\wbar)^2\leq C_P\intom |\na w|^2$ for $w\in \Wezom$. 
\end{lemma}
\begin{proof}
 As announced, this is a direct consequence of Poincar\'e's inequality:
 \[
  C_P\intom|\na w|^2\geq \intom (w-\wbar)^2=\intom w^2 - 2\intom w\wbar+\intom \wbar^2=\intom w^2-2\wbar|\Om|\wbar +|\Om|\wbar^2=\intom w^2-|\Om|\wbar^2.\qedhere
 \]
\end{proof}

Another elementary but useful identity is the following:
\begin{lemma}
\label{lem:hoheableitungen}
Let $w\in C^3(\Om)$. Then
 \[
  \Lap|\na w|^2=2\na w\na\Lap w+2|D^2 w|^2.
 \]
\end{lemma}


In the proof of Lemma \ref{lem:navsechs} we will also make use of the well-known Gagliardo-Nirenberg inequality:
\begin{lemma}
 \label{lem:gagliardonirenberg}
 Let $\Om$ be a bounded Lipschitz domain in $\Rn$, $p,q,r,s\ge 1$, $j,m\in\N_0$ and $\al\in[\frac j m,1]$ satisfying $\frac1p=\frac j m+(\frac1r-\frac mn)\al+\frac{1-\al}q$. Then there are positive constants $C_1$ and $C_2$ such that for all functions $w\in\Lqom$ with $\na w\in L^r(\Om)$, $w\in L^s(\Om)$, 
\[
 \norm[\Lpom]{D^j w}\leq C_1\norm[L^r(\Om)]{D^m w}^\al\norm[L^q(\Om)]{w}^{1-\al}+C_2\norm[L^s(\Om)]{w}.
\]
\end{lemma}
 \begin{proof}
  See \cite[p.126]{Nirenberg_59}. 
 \end{proof}

We are aiming for an estimate for $\intom |\na \ve|^4$. During the calculations we therefore will have to get rid of integrals of $|\na \ve|^6$. 
The Gagliardo-Nirenberg inequality enables us to replace them by more convenient terms. 
\begin{lemma}
 \label{lem:navsechs}
 Let $n=3$. For any $a>0$ there is $C(a)>0$ such that, for any $\kappa\in\R$, $\eps>0$, 
\[
 \intom |\na \ve|^6 \leq a\intom |\na|\na \ve|^2|^2+C(a)\left[\left(\intom|\na\ve|^4\right)^3+\left(\intom|\na\ve|^4\right)^{\frac32}\right].
\]
\end{lemma}
\begin{proof}
For given $j=0$, $m=1$, $\Om$, $p=3$, $r=q=2$, $s=2$, the Gagliardo-Nirenberg inequality (Lemma \ref{lem:gagliardonirenberg}) provides constants $C_1$ and $C_2$ such that for $w\in\Lzom$ and with $\al=\frac12$ 
the inequality 
\[
 \norm[\Ldom]{w}^3\leq 8C_1^3\norm[\Lzom]{\na w}^{\frac32}\norm[\Lzom]{w}^{\frac32}+8C_2^3\norm[\Lzom]{w}^3, 
\]
holds true (where we at the same time have used $(x+y)^3<8(x^3+y^3)$).
Applied to $w=|\na \ve|^2$ 
this means
\[
 \intom|\na\ve|^6\leq 8C_1^3\left(\intom|\na|\na\ve|^2|^2\right)^{\frac34}\left(\intom|\na\ve|^4\right)^{\frac34}+8C_2^3\left(\intom|\na\ve|^4\right)^{\frac32}.
\]
With $p=\frac43$, $q=4$, corresponding to $a>0$ Young's inequality provides $\Ctilde(a)>0$ such that 
\[
 \intom|\na\ve|^6\leq a\intom|\na|\na \ve|^2|^2+\Ctilde(a)\left(\intom|\na\ve|^4\right)^3+8C_2^3\left(\intom|\na\ve|^4\right)^{\frac32}
\]
and the claim results with the choice of $C(a)=\max\set{8C_2^3,\Ctilde(a)}$.
\end{proof}

With the help of Lemma \ref{lem:navsechs}, we separate $\ue$, $\na \ue$ and $\na\ve$ in one of the terms arising from differentiation of $\intom \ue^2$.
\begin{lemma}
 \label{lem:unaunav}
 Let $n=3$. Corresponding to $\my>0$ there exists $C>0$ such that for any $\kappa\in\R$ and $\eps>0$ the estimate
 \[
  \intom\ue\na\ue\cdot\na\ve\leq\frac14\intom |\na\ue|^2+\my\intom\ue^3+\frac12\intom|\na|\na\ve|^2|^2+C\left(\left(\intom|\na\ve|^4\right)^3+\left(\intom|\na\ve|^4\right)^{\frac32}\right)
 \]
holds. 
\end{lemma}
\begin{proof}
 Double application of Young's inequality yields a constant $\Ctilde>0$ such that 
\[
 \intom\ue\na\ue\cdot\na\ve\leq\frac14\intom|\na\ue|^2+\intom\ue^2|\na\ve|^2\leq\frac14\intom|\na\ue|^2+\my\intom\ue^3+\Ctilde\intom|\na\ve|^6.
\]
Using Lemma \ref{lem:navsechs} 
with $a=\frac12$ to estimate $\intom|\na\ve|^6$ this produces the assertion, with the choice $C=\Ctilde C(\frac12)$.
\end{proof}

The term $\intom |\na |\na \ve|^2|^2$, known to us from Lemma \ref{lem:navsechs}, arises from the following estimate with the ``right'' sign.
\begin{lemma}
\label{lem:dtnavzq}
 Let $\kappa\in\R$, let $q\geq 1$. Then 
\[
 \delt\intom|\na\ve|^{2q}\leq -q(q-1)\intom|\na\ve|^{2q-4}|\na|\na\ve|^2|^2-2q\intom|\na\ve|^{2q}+2q\intom|\na\ve|^{2q-1}|\na \ue|.
\]
\end{lemma}
\begin{proof}
Evaluating the derivative and inserting the second equation of \eqref{eq:epsprob} gives
 \[
 \delt\intom|\na\ve|^{2q} 
 =2q\intom|\na\ve|^{2q-2}\na\ve\cdot\na\Lap\ve-2q\intom|\na\ve|^{2q-2}\na\ve\cdot\na\ve+2q\intom|\na\ve|^{2q-2}\na \ve\cdot \na\ue. 
 \]
Here, Lemma \ref{lem:hoheableitungen} and integration by parts eventuate 
\begin{align*}
 \delt\intom|\na\ve|^{2q}=&q\intom|\na\ve|^{2q-2}\Lap|\na\ve|^2-2q\intom|\na\ve|^{2q-2}|D^2\ve|^2-2q\intom|\na\ve|^{2q}+2q\intom|\na\ve|^{2q-2}\na\ve\cdot\na\ue\\
\leq& -q(q-1)\intom|\na\ve|^{2q-4}|\na|\na \ve|^2|^2 - 2q\intom|\na\ve|^{2q}+2q\intom|\na\ve|^{2q-1}|\na\ue|.
\end{align*}
In this step we used convexity of $\Om$ to estimate the boundary integral 
\[
 \intdom |\na \ve|^{2q-2} \na(|\na \ve|^2)\cdot \ny\leq 0 
\]
due to the fact that in convex domains $\left.\delny |\na \ve|^2\right|_{\dom} \leq 0$ follows from $\delny v\amrand=0$, confer \cite[Lemma 3.2]{tao_winkler_12_boundedness_in_a_quasilinear}. 
\end{proof}

The other summand arising in the calculation of $y_\eps'(t)$ can be estimated as follows:
\begin{lemma}
\label{lem:dtuepsquadrat}
 For any $\kappa\in\R$, $\eps>0$, 
 \begin{align*}
  \delt\intom\ue^2\le&-2\intom|\na\ue|^2+2\intom\ue\na\ue\cdot\na\ve+2\kappa\intom\ue^2-2\my\intom\ue^3.
 \end{align*}
\end{lemma}
\begin{proof}
 This results from integration by parts and estimation of the negative last term in 
 \[
  2\intom \ue\uet \leq 2\intom \ue\Lap\ue - 2\intom \ue\na\cdot(\ue\na\ve) + 2\kappa\intom\ue^2-2\my\intom\ue^3-2\eps\intom\ue^\zv.\qedhere
 \]
\end{proof}

We put the estimates that we have found so far to their use and state

\begin{proposition}
 \label{prop:ODIulzvwev}
 Let $n=3$ and $\my>0$. There is a constant $\Ceins>0$ such that for all $\eps>0$, for all $\ny>0$, $\eta\in(0,4]$ and $\kappahat>0$ the following holds: If $\kappa\in\R$ satisfies $\kappa<\kappahat$ and $2\kappa+\eta\leq\frac1{C_P}$, where $C_P $ is the Poincar\'e-constant associated with $\Om$, then the quantity
\begin{equation}\label{eq:defyeps}
 \ye(t):= \intom \ue^2(t)+\intom|\na\ve(t)|^4
\end{equation}
satisfies the differential inequality
\[
 \ye'(t)\leq \ny - \eta y(t) + \Ceins(1+\frac{1}{4\ny}) y^3(t) + \frac{4\kappahat^2|\Om|}{C_P\my^2}=:p(\ye(t))
\]
for all $t>T_0$ with some $T_0=T_0(\my,\kappa,\kappahat)>0$ depending on $\my, \kappa, \kappahat$ only.
\end{proposition}
\begin{proof}
 With the aid of Lemma \ref{lem:ule}, fix $T_0>0$ such that 
\begin{equation}
 \label{eq:intsmall}
 \intom \ue(\tau)<\frac{2\kappahat|\Om|}{\my}\qquad \text{for all } \tau>T_0,\;\eps>0.  
\end{equation}
 By Lemma \ref{lem:dtuepsquadrat} 
 and Lemma \ref{lem:dtnavzq} 
with $q=2$, we have
\begin{align*}
 \ye'(t)=&\delt\left(\intom \ue^2+\intom|\na\ve|^4\right)\\
\leq&-2\intom|\na \ue|^2+2\intom \ue\na\ue\cdot\na\ve+2\kappa\intom\ue^2-2\my\intom\ue^3\\
&-2\intom|\na|\na\ve|^2|^2-4\intom|\na\ve|^4 + 4\intom |\na\ve|^3|\na\ue|.
\end{align*}
By application of Lemma \ref{lem:unaunav} 
to the second and Young's inequality and Lemma \ref{lem:navsechs}
to the last term, this becomes
 \begin{align*}
 \ye'(t)\leq&-2\intom|\na\ue|^2+\frac12\intom|\na\ue|^2+2\my\intom \ue^3+\intom |\na|\na \ve|^2|^2+2C\left(\left(\intom|\na\ve|^4\right)^3+\left(\intom |\na\ve|^4\right)^{\frac32} \right)\\+&2\kappa\intom \ue^2
-2\my\intom \ue^3
-2\intom|\na|\na \ve|^2|^2-4\intom |\na \ve|^4+\frac12\intom |\na \ue|^2\\
&+8\left (\frac18\intom |\na|\na\ve|^2|^2+C(\frac18)\left(\left(\intom|\na\ve|^4\right)^3+\left(\intom |\na\ve|^4\right)^{\frac32} \right)\right)\\
 \leq&2\kappa\intom \ue^2+\Ceins\left(\intom|\na\ve|^4\right)^3+\Ceins^{\frac12}\left(\intom|\na\ve|^4\right)^{\frac32}-4\intom |\na\ve|^4-\intom|\na \ue|^2,
 \end{align*}
where we denoted $\Ceins^{\frac12}=\max{2C+8C(\frac18),1}\leq \Ceins$, $C$ being the constant from Lemma \ref{lem:unaunav} and $C(\frac18)$ taken from Lemma \ref{lem:navsechs}.

Another application of Young's inequality with $\ny>0$ -- so as to remove the unsolicited exponent $\frac32$ --
and sorting other terms, in order that the term $-\eta y$ appears, leave us with 
\[
 \ye'(t)\leq (2\kappa+\eta)\intom\ue^2+\Ceins\left(\intom|\na\ve|^4\right)^3+\ny+\frac{\Ceins}{4\ny}\left(\intom|\na\ve|^4\right)^3-\eta\intom \ue^2 - 4\intom|\na\ve|^4 -\intom |\na\ue|^2, 
\]
where we apply Lemma \ref{lem:poincare} to the last summand and use that by \eqref{eq:intsmall} $\ubar_\eps(t)=\frac{1}{|\Om|}\intom \ue(t)\leq \frac{2\kappahat}{\my}$ for $t>T_0$ to arrive at 
 \begin{align*}
 \ye'(t)\leq &  \left((2\kappa+\eta)-\frac1{C_P}\right)\intom \ue^2+\Ceins\left(1+\frac{1}{4\ny}\right)\left(\intom|\na\ve|^4\right)^3+\ny-\eta \left(\intom \ue^2+\intom |\na\ve|^4\right) + \frac{|\Om|}{C_P} \ubar_\eps^2\\
 \leq& \ny-\eta \ye(t)+ \Ceins\left(1+\frac{1}{4\ny}\right) (\ye(t))^3+\frac{4|\Om|\kappahat^2}{C_P\my^2} 
 \end{align*}
as long as $(2\kappa+\eta)C_P\leq 1$ and $\eta\in(0,4]$.
\end{proof}

The function $y_\eps$ satisfies a differential inequality with polynomial right hand side. This information is not very useful in obtaining boundedness
if not accompanied by the statement that comparison with a stationary solution to the differential equation might be possible, i.e. that there is a root of the polynomial.
Such is provided by the following lemma.

\begin{lemma}
\label{lem:nullstelle}
 For any $\my>0$ there exists $\ny_0>0$ such that for all $\ny
\in(0,\ny_0]$ there are $\kappatilde>0$, $\eta\in(0,4]$ such that the polynomial 
\[
 p(x)=\ny-\eta x+\Ceins\left(1+\frac{1}{4\ny}\right) x^3+\frac{4\kappahat^2|\Om|}{C_P \my^2}
\]
defined in Proposition \ref{prop:ODIulzvwev} 
has a positive root for $\kappahat=\kappatilde$.

Furthermore, for each $\kappahat\in[0,\kappatilde]$ it has a largest positive root $\delta_\ny(\kappahat)$ as well, satisfying 
\[
 \delta_\ny(\kappatilde)\leq\delta_\ny(\kappahat)\leq \sqrt{\frac{4}{\Ceins(1+\frac{1}{4\ny})}}.
\]
\end{lemma}
\begin{proof}
 Because $p(x)$ is increasing in $\kappahat$, $\delta_\ny(\kappatilde)\leq\delta_\ny(\kappahat)$ for $\kappahat\in[0,\kappatilde]$ is obvious.

  Choose $\ny_0>0$ such that 
 \[
  \ny_0^2+\frac{\ny_0}4 <\min\setl{\frac{4}{27 \Ceins C_P^3},\frac{256}{27 \Ceins}}
 \]
 and let $\ny\in(0,\ny_0]$. Then the inequality 
 \begin{equation}
 \label{eq:nyandkappa}
  \left(\ny+\frac{4|\Om|}{\my^2 C_P}\kappahat^2\right)^2 < \frac{4\min\setl{\left(\frac1{C_P}-2\kappahat\right)^3,4^3}}{27 \Ceins(1+\frac1{4\ny})}
 \end{equation}
 is satisfied with $\kappahat=0$. Let $\kappatilde\in(0,\frac{1}{2C_P})$ be such that \eqref{eq:nyandkappa} is still satisfied for $\kappahat=\kappatilde$. This is possible due to continuity of the expressions in $\kappahat$. 
 Additionally, let $\eta=\min\set{4,\frac{1}{C_P}-2\kappatilde}$.
 
 Consequently, the inequality 
\begin{equation}
 \label{eq:pisnegative}
 \left(\ny+\frac{4|\Om|}{C_P\my^2}\kappatilde^2\right)^2< \frac{4\eta^3}{27 \Ceins(1+\frac1{4\ny})},
 \quad\text{that is}\quad\ny-\frac23\eta\sqrt{\frac{\eta}{3\Ceins(1+\frac{1}{4\ny})}}+ \frac{4|\Om|}{C_P\my^2}\kappatilde^2<0
\end{equation}
 holds.
 Observe that $p$ attains a local minimum at 
\[
 x_m=\sqrt{\frac{\eta}{3\Ceins(1+\frac{1}{4\ny})}}>0, 
\]
where 
\begin{align*}
 p(x_m)=& \ny-\eta \sqrt{\frac{\eta}{3\Ceins(1+\frac{1}{4\ny})}} + \Ceins(1+\frac{1}{4\ny})\frac{\eta}{3\Ceins(1+\frac{1}{4\ny})}\sqrt{\frac{\eta}{3\Ceins(1+\frac{1}{4\ny})}}+\frac{4|\Om|}{C_P\my^2}\kappatilde^2\\
=& \ny-\frac23\eta\sqrt{\frac{\eta}{3\Ceins(1+\frac{1}{4\ny})}}+ \frac{4|\Om|}{C_P\my^2}\kappatilde^2
\end{align*}
is negative by \eqref{eq:pisnegative} and therefore $p$ has a root in $(x_m,\infty)$.
For any $\kappahat\in[0,\kappatilde]$ this root is smaller than $\sqrt{\frac4{\Ceins(1+\frac4\ny)}}$, because for $x>\sqrt{\frac4{\Ceins(1+\frac4\ny)}}>\sqrt{\frac\eta{\Ceins(1+\frac4\ny)}}$ we have 
\[
 p(x)>\Ceins(1+\frac1{4\ny})x^3-\eta x\geq 0.\qedhere
\]
\end{proof}

We use this root for a comparison argument:
\begin{proposition}
 \label{prop:staysmall}
 Let $n=3$ and $\my>0$, let $\ny_0$, $\eta$, $\kappatilde$ and $\delta_\ny(\kappatilde)$ for some $\ny\in(0,\ny_0]$ be as in Lemma \ref{lem:nullstelle}. Then for any $0\leq\kappahat\leq\kappatilde$, $\delta_\ny(\kappahat)>\delta_\ny(\kappatilde)>0$ is such that for every $\kappa\leq\kappahat$ every $\eps>0$ has the following property: If $\ye$ from \eqref{eq:defyeps} satisfies
\[
 \ye(T)\le\delta_\ny(\kappa)
\]
for some $T>T_0$ (with $T_0=T_0(\my,\kappa,\kappahat)$ from Proposition \ref{prop:ODIulzvwev}), then $\ye(t)\le\delta_\ny(\kappa)$ for all $t>T$.
\end{proposition}
\begin{proof}
 Choose as $\delta=\delta_\ny(\kappahat)$ the largest root of $p$ from Lemma \ref{lem:nullstelle} and observe that according to Proposition \ref{prop:ODIulzvwev} and the assumption on $T$
\[
 \ye'(t)\leq p(y(t)) \qquad\text{for all \ }t>T \qquad\text{and}\qquad \ye(T)\leq\delta.
\]
The comparison principle for ordinary differential equations therefore shows by means of comparison with $\ybar\equiv\delta$ that $\ye(t)\leq \delta$ for all $t>T$ as well.
\end{proof}

\subsection{Eventual boundedness of $y_\eps$}
Proposition \ref{prop:staysmall} asserts that $y_\eps$ stays small, should it ever fall below a certain value. 
We still have to ensure that the condition actually occurs.

\begin{proposition}
 \label{prop:wirdmalklein} Let $n=3$. Let $\ny\in(0,\ny_0]$ with $\ny_0$ as in Lemma \ref{lem:nullstelle}. Then there exists $\kappa_0\in(0,\frac18)$ 
such that for any $\kappa<\kappahat\in(0,\kappa_0]$ there is $t_0>0$ such that for all $\tau>t_0$, for all $\eps>0$
 \[\intom \ue^2(\tau)+|\na\ve(\tau)|^4<\delta_\ny(\kappahat)\]
 where $\delta_\ny(\kappahat)>0$ is the positive root of $p$ given by Lemma \ref{lem:nullstelle}.\\
 Furthermore, $\kappahat$ satisfies 
\begin{equation}\label{eq:kappadelta}
 \kappahat\leq \sqrt{\frac{\delta_\ny(\kappahat)\my^2}{(4+8C_\Om)|\Om|}}, 
\end{equation}
 where $C_\Om$ is a constant depending on the domain $\Om$ only. 
\end{proposition}
\begin{proof}
 Due to the embedding $W^{2,2}(\Om)\embeddedinto W^{1,4}(\Om)$
, there is $C_\Om>0$ such that 
\begin{equation}
 \label{eq:embeddingw22w14}
\intom |\na w|^4\leq C_\Om\intom(w^2+|\Lap w|^2)
\end{equation}
for all $w\in W^{2,2}(\Om)$.

Let $\ny$ be as given in the statement of the proposition, let $\kappatilde>0$ be as provided by Lemma \ref{lem:nullstelle} and let $\delta=\delta_\ny(\kappatilde)$. 
Choose 
\begin{equation}
\label{eq:defkappa0}
 0<\kappa_0<\min\setl{\kappatilde,\sqrt{\frac{\delta\my^2}{(4+8C_\Om)|\Om|}},\frac18}
\end{equation}
and let $\kappahat\in(0,\kappa_0]$ and $\kappa<\kappahat$. (This already ensures \eqref{eq:kappadelta} as well as the applicability of Proposition \ref{prop:staysmall}.)

Let $T_0=T_0(\my,\kappa,\kappahat)$ be as provided by Proposition \ref{prop:ODIulzvwev} and let $t>T_0$. Note that as a result of \eqref{eq:intsmall} this entails 
\begin{equation}\label{eq:intklein_wieder}\intom \ue(t)<\frac{2\kappahat|\Om|}{\my}.\end{equation}

Furthermore denote 
\begin{equation}
 \label{eq:c0}
 C_0=\max\setl{ 1+\int|\na v_0|^2+\frac{1}{\my}\intom u_0+\frac1{\my},\frac{\kappahat+1}{\my} \max\setl{1+\intom u_0, \frac{\kappahat|\Om|}{\my}}}
\end{equation}
and choose $T>0$ so large that  
\begin{align}
\label{eq:defT}
 \frac1T\bigg(&
\frac{2\kappahat|\Om|}{\my^2}+\frac{2C_\Om\kappahat|\Om|}{\my^2}
+C_\Om\frac{\kappahat}{\my} \max\setl{1+\intom u_0, \frac{\kappahat|\Om|}{\my}} t\nonumber\\
&\;\,+C_\Om\frac1{\my}\intom u_0+\frac{C_\Om}{\my}+C_\Om\intom v_0^2+C+2CC_0 +\frac{2C_\Om\kappahat|\Om|}{\my^2}
\bigg)<\frac\delta2.
\end{align}
Combining \eqref{eq:embeddingw22w14} with Lemmata \ref{lem:ulz}, \ref{lem:vlevlz} and \ref{lem:vderivatives} gives 
\begin{align*}
 \int_t^{t+T}\intom (\ue^2+|\na\ve|^4)\leq&\int_t^{t+T}\intom\ue^2+C_\Om\int_t^{t+T}\intom\ve^2+C_\Om\int_t^{t+T}\intom |\Lap\ve|^2\\
\leq&\frac{\kappa_+}{\my} \max\setl{\intom \ue(t), \frac{\kappa_+|\Om|}{\my}} T+\frac1{\my}\intom \ue(t)\\ 
&+ C_\Om \frac{\kappa_+}{\my} \max\setl{\intom \ue(t), \frac{\kappa_+|\Om|}{\my}} T+C_\Om\frac1{\my}\intom \ue(t)+C_\Om\intom \ve^2(t)\\
&+C_\Om \frac{\kappa_+}{\my}\max\setl{\intom \ue(t), \frac{\kappa_+|\Om|}{\my}}T+C_\Om\intom |\na \ve(t)|^2 +\frac{C_\Om}{\my} \intom u(t).
\end{align*}
Due to \eqref{eq:intklein_wieder}, upon another application of Lemmata \ref{lem:vlevlz} and \ref{lem:vderivatives} and taking into account 
that $\kappa_+\leq\kappahat$, this reduces to 
\begin{align*}
 \int_t^{t+T}\intom (\ue^2+|\na\ve|^4)
\leq&\frac\kappahat{\my} \frac{2\kappahat|\Om|}{\my} T+\frac{2\kappahat|\Om|}{\my^2}\\ 
&+ C_\Om \frac{\kappahat}{\my} \frac{2\kappahat|\Om|}{\my} T+\frac{2C_\Om\kappahat|\Om|}{\my^2}+C_\Om\frac{\kappahat}{\my} \max\setl{1+\intom u_0, \frac{\kappahat|\Om|}{\my}} t\\&+C_\Om\frac1{\my}\left(1+\intom u_0\right)+C_\Om+C_\Om\intom v_0^2
+C_\Om \frac{\kappahat}{\my}\frac{2\kappahat|\Om|}{\my}T+2C_\Om C_0 +\frac{2C_\Om\kappahat|\Om|}{\my^2},
\end{align*}
where $C_0$ is as defined in \eqref{eq:c0}. Therefore, 
\begin{align*}
 \frac1T\int_t^{t+T}\intom (\ue^2+|\na\ve|^4)\leq &
\frac{(2+4C_\Om)|\Om|}{\my^2}\kappahat^2\\
+&\frac1T\bigg(
\frac{2\kappahat|\Om|}{\my^2}+\frac{2C_\Om\kappahat|\Om|}{\my^2}
+C_\Om\frac{\kappahat}{\my} \max\setl{1+\intom u_0, \frac{\kappahat|\Om|}{\my}} t\\
&\;\,+C_\Om\frac1{\my}\intom u_0+\frac{C_\Om}{\my}+C_\Om\intom v_0^2+C_\Om+2C_\Om C_0 +\frac{2C_\Om\kappahat|\Om|}{\my^2}
\bigg).
\end{align*}
Our choice of $\kappa_0$ and $T$ in \eqref{eq:defkappa0} and \eqref{eq:defT}, respectively, now entails 
\[
 \frac1T\int_t^{t+T}\intom (\ue^2+|\na\ve|^4)\leq \delta(\kappatilde). 
\]
Accordingly, for at least one $t_0\in(t,t+T)$ 
\[
 \intom (\ue^2(t_0)+|\na\ve(t_0)|^4)\leq \delta(\kappatilde)
\]
holds as well. Due to $\delta_\ny(\kappatilde)\leq \delta_\ny(\kappa_0)\leq\delta_\ny(\kappahat)$ for $0<\kappahat<\kappa_0$, the claimed inequality for larger times $\tau$ is a direct consequence of Proposition \ref{prop:staysmall}.
\end{proof}

\subsection{Eventual boundedness of $(u,v)$ in $\Liom\times W^{1,\infty}(\Om)$}
The next step is to refine these bounds on $y_\eps$ to bounds on $u$, $v$ and $\na v$. $L^p-L^q$ estimates for the heat semigroup will be the cornerstone of this procedure.

\begin{proposition}
 \label{prop:uepsfinallybounded} Let $n=3$. Then there exists a function $K\col [0,\infty)\to[0,\infty)$ satisfying $\lim_{\delta\to 0} K(\delta)=0$ with the following properties: 

Assume, $\ny\in(0,\ny_0)$, $\kappa<\kappa_0$ with $\ny_0$, $\kappa_0$ as in Lemma \ref{lem:nullstelle} and Proposition \ref{prop:wirdmalklein} respectively. Choose $\kappahat\in(\kappa_+,\kappa_0]$ and let $\delta_\ny(\kappahat)$ be as given by Proposition \ref{prop:wirdmalklein}. 
 Then there are $T_*>0$ and $C>0$ such that for all $t>T_*$ and for all $\eps>0$
\[
 \norm[\Liom]{\ue(t)}+\norm[\Liom]{\na\ve(t)}<K(\delta_\ny(\kappahat)).
\]
Furthermore, corresponding to any $\norm[\Lzom]{u_0}$, $\norm[\Lzom]{v_0}$, there is $C>0$ such that for all $t>T_*$, $\eps>0$ 
\[
 \norm[\Liom]{\ve(t)}\leq Ce^{-(t-T_*)}+K(\delta_\ny(\kappahat)).
\]
\end{proposition}
\begin{proof}
Let $\ny_0, \ny, \kappa\leq\kappa_+<\kappahat<\kappa_0$ and $\delta:=\delta_\ny(\kappahat)$ be as indicated in the statement of the proposition.
Let $T_*-2$ be the number from Proposition \ref{prop:wirdmalklein}, let $t_0\geq T_*-2$ and let us first show boundedness of $\ue$ on $[t_0+1,t_0+2]$. 
Define 
\[
 M:=\sup_{\tau\in(t_0,t_0+2]} \norm[\Liom]{(\tau-t_0)^{\frac34} \ue(\tau)}.
\]

Let $p\in(3,4)$.
By the choice of $T_*$ and $\delta$ and Proposition \ref{prop:wirdmalklein}, $\intom u^2(t)+\intom  |\na \ve(t)|^4\leq \delta$ for $t>T_*$. Together with H\"older's inequality this implies 
\begin{align}
\label{eq:integrand}
 \norm[\Lpom]{\ue(s)\na \ve(s)}\leq&
\norm[L^{\frac{4p}{4-p}}(\Om)]{\ue(s)}\norm[L^4(\Om)]{\na\ve(s)}\leq\delta^{\frac14} \left(\intom \ue^{\frac{4p}{4-p}}(s)\right)^{\frac{4-p}{4p}}\ \notag\\
\leq&\delta^{\frac14}\left(\intom \ue^2(s)\cdot \ue^{\frac{4p}{4-p}-2}(s)\right)^{\frac{4-p}{4p}}\notag
\leq\delta^{\frac14}\left(\intom \ue^2(s)\right)^{\frac{4-p}{4p}}\left(\sup_{\Om} \ue^{\frac{4p}{4-p}-2}(s)\right)^{\frac{4-p}{4p}}\notag\\
\leq&\delta^{\frac14+\frac{4-p}{4p}} \sup_\Om(\ue^{1-\frac{4-p}{2p}}(s)) 
\leq\delta^{\frac{1}{p}} (s-t_0)^{-\frac34+\frac34\frac{4-p}{2p}} \sup_\Om((s-t_0)^{\frac34}\ue(s))^{1-\frac{4-p}{2p}}
\end{align}
for $s\in(t_0,t_0+2]$. 
Triangle inequality and $L^p-L^q-$estimates \cite[Lemma 1.3]{winkler_10_aggregationvs} give a constant $C_1>0$ such that 
\newcommand{\uebar}{\ubar_\eps}

\begin{equation*}
 \norm[\Liom]{e^{\tau\Lap}\ue(t_0)}\leq C_1 (1+\tau^{-\frac34})\norm[\Lzom]{\ue(t_0)-\uebar(t_0)}+\norm[\Liom]{\uebar(t_0)},
\end{equation*}
where $\norm[\Liom]{\ubar_\eps(t_0)}=\frac1{|\Om|}\intom \ue(t_0)\leq |\Om|^{-\frac12}\norm[\Lzom]{\ue(t_0)}\leq |\Om|^{-\frac12}\sqrt{\delta}$ and thus 
$\norm[\Lzom]{\uebar(t_0)}\leq\sqrt{\delta}$ lead to 
\begin{equation}
\label{eq:uliom} 
 \norm[\Liom]{e^{\tau\Lap}\ue(t_0)}\leq C_1 (1+\tau^{-\frac34})2\sqrt{\delta}+|\Om|^{-\frac12}\sqrt{\delta}. 
\end{equation}

Again, by semigroup representation and the fact that the heat semigroup is order-preserving, 
\begin{align*}
 0\leq\tau^{\frac34} \ue(t_0+\tau)\leq&\tau^{\frac34}e^{\tau\Lap}\ue(t_0) - \tau^{\frac34}\intntau e^{(\tau-s)\Lap}\na\cdot(\ue(t_0+s)\na \ve(t_0+s))\ds \\&+\tau^{\frac34}\intntau e^{(\tau-s)\Lap}\left(\kappa_+ \ue(t_0+s)-\my \ue^2(t_0+s)\right)\ds \\
\leq& \tau^{\frac34}\norm[\Liom]{e^{(\tau-s)\Lap}\ue(t_0)}  + \tau^{\frac34}\intntau \norm[\Liom]{e^{(\tau-s)\Lap}\na\cdot(\ue(t_0+s)\na \ve(t_0+s))}\ds\\ &+ \tau^{\frac34}\intntau \kappa_+ M s^{-\frac34} \ds.
\end{align*}

Together with $L^p-L^q-$estimates, \eqref{eq:integrand} and \eqref{eq:uliom} this entails for $\tau\in[0,2]$ and some $C_2>0$ from \cite[Lemma 1.3]{winkler_10_aggregationvs}

 \begin{align*}
  \norm[\Liom]{\tau^{\frac34} \ue(t_0+\tau)}\leq& \tau^{\frac34}C_1 (1+\tau^{-\frac34})2\sqrt{\delta}+\tau^{\frac34}|\Om|^{-\frac12}\sqrt{\delta}+8\kappa_+ M\\
&+]2\int_{0}^{\tau}(1+(\tau-s)^{-\frac12-\frac3{2p}})\norm[\Lpom]{\ue(t_0+s)\na\ve(t_0+s)} ds\\
 \leq& 2C_1 (1+\tau^{\frac34})\sqrt{\delta}+\tau^{\frac34}|\Om|^{-\frac12}\sqrt{\delta}+8\kappa_0 M\\
 &+C_2\int_0^\tau(1+(\tau-s)^{-\frac12-\frac3{2p}}) \delta^{\frac{1}{p}}s^{-\frac34+\frac34\frac{4-p}{2p}}\norm[\Liom]{s^{\frac34}\ue(t_0+s)}^{1-\frac{4-p}{2p}} ds \\
 \leq&2C_1 (1+\tau^{\frac34})\sqrt{\delta}+\tau^{\frac34}|\Om|^{-\frac12}\sqrt{\delta}+8\kappa_0 M\\
 &+C_2\int_0^\tau(1+(\tau-s)^{-\frac12-\frac3{2p}}) \delta^{\frac{1}{p}}s^{-\frac34+\frac34\frac{4-p}{2p}}M^{1-\frac{4-p}{2p}} ds,  
 \end{align*}

%
As
\(
 \int_0^2(1+(\tau-s)^{-\frac12-\frac3{2p}})s^{-\frac34+\frac34\frac{4-p}{2p}}\ds 
\)
is finite and $\frac{1}{1-8\kappa_0}>0$, 
taking the supremum over $\tau\in[0,2]$, we infer
\[
 M\leq C_3 \sqrt{\delta} + C_4 \delta^{\frac1p} M^{1-\frac{4-p}{2p}}
\]
with obvious choices of the constants $C_3,C_4>0$.

%
Therefore 
\[
 M\leq D(\delta):=\sup\set{\xi\in[0,\infty):\xi-C_4\delta^{\frac1p} \xi^{1-\frac{4-p}{2p}}\leq C_3\sqrt{\delta}}<\infty.
\]
Note that $D(\delta)$ tends to $0$ as $\delta$ becomes small.

For $t\in[t_0,t_0+2]$ 
\[
 (t-t_0)^{\frac34}\norm[\Liom]{\ue(t)}<D(\delta),
\]

meaning that for $t\in[t_0+1,t_0+2]$
\[
 \norm[\Liom]{\ue(t)}<D(t-t_0)^{-\frac34}\leq D(\delta).
\]
$D(\delta)$ is independent of the choice of $t_0>T_*-2$, therefore we can conclude 
\begin{equation}
\label{eq:boundu}
 \norm[\Liom]{\ue(t)}\leq D(\delta) 
\end{equation}
for any $t>T_*-1$.

Boundedness of $\set{\na v(\tau)}_{\tau >T_*}$ in $\Liom$ can be achieved from the following estimates:
Let $t_0=T_*-1$ and denote $t=\tau-t_0$. Then \cite[Lemma 1.3]{winkler_10_aggregationvs} provides $C_5>0$ such that
\begin{align}\label{eq:boundnav}
 \norm[\Liom]{\na \ve(\tau)}\leq&\norm[\Liom]{\na e^{t(\Lap-1)}\ve(t_0)}+\intnt \norm[\Liom]{\na e^{(t-s)(\Lap-1)}\ue(t_0+s)}\ds\nonumber\\
\leq& C_5 t^{-\frac78}\norm[L^4(\Om)]{\na \ve(t_0)}+ C_5 \intnt (1+(t-s)^{-\frac12}) e^{-(t-s)} \norm[\Liom]{\ue(t_0+s)}\ds\nonumber\\
\leq& C_5 \delta^{\frac14} t^{-\frac78} + D(\delta) C_5 \intninf (1+\sigma^{-\frac12})e^{-\sigma} d\sigma
\end{align}
 is bounded on $[T_*,\infty)$. By similar reasoning together with Lemma \ref{lem:vlevlz}, we obtain bounds on $\norm[\Liom]{\ve(\tau)}$. In preparation for these estimates, let $t_*>t_0$ and let us note that by \eqref{eq:intvdecay} and Lemma \ref{lem:vlevlz} 
\begin{align*}
  \frac1{|\Om|}\intom \ve(t_*)\leq& \frac1{|\Om|}\left(\intom \ve(t_0)\right) e^{-(t_*-t_0)}+\frac{\kappa_+}{\my}+\frac1{|\Om|}\intom \ue(t_0)\\
\leq& (\frac{\kappa_+}\my+ C_6)e^{-(t_*-t_0)}+\frac{\kappa_+}\my+\frac1{|\Om|^{\frac12}}\norm[\Lzom]{\ue(t_0)}\\
\leq&C_6 e^{-(t_*-t_0)}+2\frac{\kappahat}{\my}+\frac{\delta^{\frac12}}{|\Om|^{\frac12}}\\
\leq&C_6 e^{-(t_*-t_0)}+C_7 \delta^{\frac12}, 
\end{align*}
where $C_6$ depends on $\norm[L^1(\Om)]{u_0}$ and $\norm[L^1(\Om)]{v_0}$ (and $|\Om|$) only, and where we have applied \eqref{eq:kappadelta} in the last step, so that $C_7=(1+\frac1{\sqrt{(4+8C_\Om)}})\frac1{\sqrt{|\Om|}}$ with $C_\Om$ as in \eqref{eq:kappadelta}.

Lemma 1.3 of \cite{winkler_10_aggregationvs} yields $C_8>0$, which, in conjunction with Poincar\'e's inequality and \eqref{eq:boundu}, gives 
\begin{align*}
 \norm[\Liom]{\ve(t_*+t)}\leq& \norm[\Liom]{e^{t(\Lap-1)}\ve(t_*)}+\intnt\norm[\Liom]{e^{(t-s)(\Lap-1)}\ue(t_*+s)}\ds\\
\leq& \norm[\Liom]{e^{t(\Lap-1)}(\ve(t_*)-\frac{1}{|\Om|}\intom \ve(t_*))} + \frac{1}{|\Om|}\intom \ve(t_*) + tD(\delta)\\
 \leq& C_8(1+t^{-\frac34})\norm[\Lzom]{\ve(t_*)-\frac{1}{|\Om|}\intom \ve(t_*)}+\frac{1}{|\Om|}\intom \ve(t_*) + 2D(\delta)\\
\leq& C_8(1+t^{-\frac34})C_P\norm[\Lzom]{\na \ve(t_*)}+\frac{1}{|\Om|}\intom \ve(t_*) + 2D(\delta)\\
\leq& C_8(1+t^{-\frac34})C_P|\Om|^{\frac14}\delta^{\frac14}+C_6e^{-(t_*-t_0)}e^{2-t}+C_7 \delta^{\frac12} + 2D(\delta)
\end{align*}
for any $t\in(0,2]$ and therefore 
\begin{equation}\label{eq:boundv}
 \norm[\Liom]{\ve(\tau)}\leq 2C_8 C_P|\Om|^{\frac14}\delta^{\frac14}+C_7 \delta^{\frac12} + 2D(\delta)+C_6ee^{-(\tau-(t_0+1))}
\end{equation}
for any $\tau>t_0+1=T_*$.
Collecting terms from \eqref{eq:boundu}, \eqref{eq:boundnav} and \eqref{eq:boundv}, we obtain a suitable definition of $C$ and of $K(\delta)$ -- and as $\delta^{\frac14}$, $\delta^{\frac12}$ and $D(\delta)$ tend to $0$ as $\delta\downto 0$, indeed, $\lim_{\delta\downto 0} K(\delta)=0$.
\end{proof}

\section{Definition of solutions}
\begin{definition}
\label{def:soln}
A pair of functions 
$(u,v)\in L^2_{loc}((0,\infty); \Lzom)\times L^2_{loc}((0,\infty); \Wezom)$ is called {\bf weak solution} of \eqref{eq:limprob} for initial data $(u_0,v_0)\in \Lzom\times \Wezom$ if for all test functions $\phii\in \Cninf(\Ombar\times[0,\infty))$ the following holds: 
\begin{equation}
\label{eq:solndefu}
 - \intninf\intom u \phii_t -\intom u_0\phii(0) = \intninf\intom u\Lap \phii-\intninf\intom u\na v\cdot\na \phii+\kappa\intninf\intom u\phii-\my\intninf\intom u^2\phii 
\end{equation}
 and, for all $\phii\in \Cninf(\Ombar\times[0,\infty))$,
 \begin{equation}
 \label{eq:solndefv}
  -\intninf\intom v\phii_t-\intom v_0\phii(0) = -\intninf\intom \na v\cdot\na\phii-\intninf\intom v\phii+\intninf\intom u\phii.
 \end{equation}
\end{definition}
%

\section{Convergence to a solution}
\label{sec:soln}
Purpose of the estimates from section \ref{sec:estimates} was to make the extraction of convergent sequences of approximate solutions $(\ue,\ve)$ possible. The following proposition lists, in which sense we have obtained convergence.
\begin{prop}
 \label{prop:conv}
 There exist $u\in L^2_{loc}((0,\infty),\Lzom)$ and $v\in L^2_{loc}((0,T),\Wezom)$ and a sequence $\eps_j\downto 0$ such that for any $T>0$
\begin{align}
 \label{conv:uae}\uej\to u&\quad\text{a.e. in }\Om\times[0,T],\\
 \label{conv:ulz}\uej\to u&\quad\text{ in } L^2(\Om\times(0,T)),\\
 \label{conv:uzv}\eps_j\uej^\zv\weakto 0&\quad\text{ in }L^1(\Om\times(0,T)),\\
 \label{conv:vlzw}\vej\weakto v&\quad\text{ in } L^2((0,T);\Wezom),\\
 \label{conv:vlz}\vej\to v&\quad\text{ in } L^2(\Om\times(0,T)),\\
 \label{conv:vae}\vej\to v&\quad \text{ a.e. in } \Om\times[0,T],\\
 \label{conv:Lapv}\Lap \vej \weakto \Lap v &\quad\text{ in } L^2(\Om\times(0,T)),\\
 \label{conv:vtwlz} v_{\eps_jt}\weakto v_t&\quad\text{ in }L^2(\Om\times(0,T)),\\
 \label{conv:unav}\uej\na\vej\weakto u\na v&\quad\text{ in }L^1(\Om\times(0,T)).
\end{align}
\end{prop}
\begin{proof}
 Lemmata \ref{lem:uinlewevd} and \ref{lem:uetbd} show boundedness of $\set{\ue}_{\eps}$ in $L^{\frac43}((0,T);W^{1,\frac43}(\Om))$ and of the derivatives $\set{\uet}_\eps$ in $L^1((0,T); (W^{2,\infty}(\Om))^*)$ so that by a variant of the Aubin-Lions-Lemma \cite[Prop. 6]{chen_juengel_liu_14}, 
 $\set{\ue}_\eps$ is relatively compact in $L^{\frac43}(\Om\times(0,T))$; in particular, there is a sequence $\eps_j\downto0$ (of which we will, without relabeling, choose further subsequences in the following) such that $\uej\to u$ almost everywhere in $\Om\times(0,T)$ for some $u\in L^{\frac43}(\Om\times(0,T))$.
 Boundedness of $\set{\ue}_\eps$ in $L^2(\Om\times(0,T))$ due to Lemma \ref{lem:ulz} 
 yields a subsequence along which $\uej\weakto u$ in $L^2(\Om\times(0,T))$.

 By Lemma \ref{lem:equi}, $\set{\ue^2}_\eps$ is equi-integrable and thus, according to \cite[Thm. IV.8.9]{dunford_schwartz_I_58}, weakly sequentially precompact in $L^1(\Om\times(0,T))$. Along a subsequence, $\uej^2\weakto u^2$ in $L^1(\ntom)$ 
and hence 
\[
 \norm[L^2(\ntom)]{\uej}^2=\int_\ntom\uej^2\cdot1\to \int_\ntom u^2\cdot1=\norm[L^2(\ntom)]{u}^2.
\]
The combination of $\uej\weakto u$ in $L^2(\ntom)$ and $\norm[L^2(\ntom)]{\uej}\to\norm[L^2(\ntom)]{u}$ shows that actually \eqref{conv:ulz} holds. 

Similarly, we see that $\set{\eps\ue^\zv}_\eps$ is equi-integrable (Lemma \ref{lem:equi}) and hence is 
weakly convergent along a subsequence. Pointwise a.e. convergence of $\ue^\zv$ to $u^\zv$ identifies the weak limit of $\eps_j\uej^\zv$ as $0$, which is \eqref{conv:uzv}.

According to Lemmata \ref{lem:vlevlz} and \ref{lem:vderivatives}, 
$\set{\ve}_\eps$ is bounded in $L^\infty((0,T);\Wezom)\embeddedinto L^2((0,T);\Wezom)$ and a subsequence with \eqref{conv:vlzw} can be found.

Furthermore, $\set{\vet}_\eps=\set{\Lap\ve-\ve+\ue}_\eps$ is bounded in $L^2((0,T);\Lzom)$ due to Lemmata \ref{lem:vderivatives}, \ref{lem:vlevlz}, \ref{lem:ulz}, and the Aubin-Lions lemma yields \eqref{conv:vlz} as well as, along another subsequence, \eqref{conv:vae}. At the same time, we can conclude \eqref{conv:Lapv} and \eqref{conv:vtwlz}.

The statement \eqref{conv:unav}, finally, results from a combination of \eqref{conv:ulz} and \eqref{conv:vlzw}.
\end{proof}

From now on, by $(u,v)$ we will denote the limit provided by Proposition \ref{prop:conv}. 
Of course, it would be desirable for $(u,v)$ to be a solution to the original problem. 
That is the case.

\begin{lemma}
 \label{lem:limesissoln}
 $(u,v)$ is a solution to \eqref{eq:limprob} in the sense of Definition \ref{def:soln}.
\end{lemma}
\begin{proof}
 Take $\phii$ as specified in Definition \ref{def:soln} and test the equations of \eqref{eq:epsprob} against it. The convergence results of Proposition \ref{prop:conv} then produce \eqref{eq:solndefu} and \eqref{eq:solndefv}.
\end{proof}

\begin{remark}
 None of the arguments used for Proposition \ref{prop:conv} 
 and Lemma \ref{lem:limesissoln} depend on dimension $n$ nor on the specific values of $\my>0$, $\kappa\in\R$.
\end{remark}

\section{Eventual smoothness. Proof of Theorem \ref{thm:eins}}
\label{sec:smooth}
In the most important scenario of spatial dimension $3$, we can show that these solutions are not only solutions in some weak sense, but
possess the property of eventual smoothness: From some time on, they are classical solutions.
Our preparations from Section \ref{sec:small} that have provided boundedness of $(\ue,\ve)$ are the first step.

The next proposition transfers these properties to $(u,v)$.
\begin{proposition}
 \label{prop:limsolnfinallybounded}
Let $n=3$ and assume, $\kappa<\kappa_0$ with $\kappa_0$ from Proposition \ref{prop:wirdmalklein}. 
 With $T_*$ denoting the number from Proposition \ref{prop:uepsfinallybounded},
 \[
  u\in L^2_{loc}([T_*,\infty),\Wezom),\qquad u_t\in L^2_{loc}([T_*,\infty),(\Wezom)^*).
 \]
 Furthermore $u,v,\na v \in L^\infty(\Om\times[T_*,\infty))$.
\end{proposition}
\begin{proof}
 On the interval $[T_*,\infty)$, from Proposition \ref{prop:uepsfinallybounded} we obtain boundedness of $\set{\ue}_{\eps\in(0,1)}$,$\set{\ve}_{\eps\in(0,1)}$ and $\set{|\na\ve|}_{\eps\in(0,1)}$ in $L^\infty(\Om\times[T_*,\infty))$ and hence can choose a sequence $\eps_j\downto 0$ such that $\ue, \ve, \na\ve$ are weak-*-convergent in this space.
 For $T>0$, boundedness of $\set{\ue}_\eps$ and $\set{\uet}_\eps$ in $L^2_{loc}([T_*,T_*+T],\Wezom)$ and $L^2_{loc}([T_*,T_*+T],(\Wezom)^*)$ respectively are guaranteed by Lemma \ref{lem:ueinwezifbd}
 and \ref{lem:uetlzwezdual} and the choice of a weakly convergent subsequence yields the assertion.
\end{proof}

\begin{corollary}
 Under the conditions of Proposition \ref{prop:limsolnfinallybounded}, 
 \label{cor:zeidler}
 $u\in C_{loc}([T_*,\infty),\Lzom)$.
\end{corollary}
\begin{proof}
 For any $T>0$, $u\in L^2([T_*,T_*+T],\Wezom)$ and $u_t\in L^2([T_*,T_*+T],(\Wezom)^*)$. 
By Proposition 23.23 of \cite{zeidler_IIA}, $u$ is $L^2$-continuous on $[T_*,T_*+T]$.
\end{proof}

Actually, $u$ and $v$ are even H\"older continuous.
\begin{lemma}
 \label{lem:holdercts}
 Let $n=3$. Assume, $\kappa<\kappa_0$ with $\kappa_0$ from Proposition \ref{prop:wirdmalklein} and let $T_*$ be as in Proposition \ref{prop:uepsfinallybounded}. 
 There is $\al>0$ such that $u,v\in C^{\al,\frac\al2}_{loc}(\Ombar\times[T_*+1,\infty))$.
 Moreover, there is $C>0$ such that for every $T>T_*+1$, 
\[
 \norm[{C^{\al,\frac\al2}(\Ombar\times[T,T+1])}]{u}+\norm[{C^{\al,\frac\al2}(\Ombar\times[T,T+1])}]{v}\leq C.
\]
\end{lemma}
\begin{proof}
 Let $T_*$ be as in Proposition \ref{prop:uepsfinallybounded} and let $t\geq T_*$ such that $\norm[\Liom]{u(t)}\leq \norm[L^\infty(\Om\times(T_*,\infty))]{u}$ and $\norm[\Liom]{v(t)}\leq \norm[L^\infty(\Om\times(T_*,\infty))]{v}$, which is the case for almost every such $t$.

 Definition \ref{def:soln}, Corollary \ref{cor:zeidler} and Proposition \ref{prop:limsolnfinallybounded} enable us to interpret $u$ as a local weak solution in the sense of \cite{porzio_vespri_93} of the equation 
 \begin{equation}
\label{eq:u_regularitaet}
  \utilde_t-\na\cdot( \na\utilde-\utilde\na v) = \kappa u-\my u^2, 
 \end{equation}
 for $\utilde$ on $[T_*,\infty)$. 

 Using boundedness of $\kappa u-\my u^2$ and $\na v$, an application of Theorem 1.3 of \cite{porzio_vespri_93} ensures $u\in C^{\al',\frac{\al'}2}(\Ombar\times[T_*+\frac12,\infty))$ for some $\al'>0$.

Theorem 1.3 of \cite{porzio_vespri_93} additionally asserts that the norm $\norm[{C^{\al,\frac\al2}(\Ombar\times[t+\frac12,t+2])}]{u}$ can be estimated by a constant $C_u$ which depends on the $\Liom$-norm of $u(t)$ and some ``data'' of the problem, a term condensing structural information on the equation (such as exponents) and certain $L^r$-norms of coefficients and the right-hand-side in \eqref{eq:u_regularitaet}.

Important to note is that, due to Proposition \ref{prop:limsolnfinallybounded}, $u, v, \na v \in L^\infty(\Om\times[T_*,\infty))$ and therefore the restrictions of these functions to $\Ombar\times[t,t+2]$ are bounded in $L^\infty(\Ombar\times[t,t+2])$ independently of $t>T_*$.
Hence $C_u$ can be chosen independently of $t$.

Similar to Corollary \ref{cor:zeidler}, from \eqref{conv:vlz}, \eqref{conv:vtwlz} and \eqref{conv:vlzw}, we infer $v\in L^2_{loc}((0,\infty),\Wezom)\cap C((0,\infty),\Lzom)$ and  boundedness of $u,v$ on $[T_*+\frac12,\infty)$ imply, again by Theorem 1.3 of \cite{porzio_vespri_93} applied to the solution $v$ of 
\begin{equation}
 \label{eq:v_regularitaet}
 \vtilde_t - \na\cdot(\na \vtilde) = u-v
\end{equation}
for $\vtilde$, that $v\in C^{\al'',\frac{\al''}2}(\Ombar\times[t+1,t+2])$ for some $\al''>0$ -- and that 
\[
 \norm[{C^{\al'',\frac{\al''}2}(\Ombar\times[t+1,t+2])}]{v}\leq C_v,
\]
with some constant $C_v$ which can be chosen independently of $t$.

Letting $\al=\min\set{\al',\al''}$, deriving a suitable constant $C$ from the values of $C_u$ and $C_v$ and taking the arbitrariness of $t$ into account, the claim follows.
\end{proof}

Thanks to the regularity of $u$ and $v$ that we have gained so far, we can interpret $u$ and $v$ as generalized solutions in the sense of \cite{LSU} of
the homogeneous Neumann boundary value problem with initial value $u(T_*+1)$, $v(T_*+1)$ to \eqref{eq:u_regularitaet} or \eqref{eq:v_regularitaet}.
As the coefficients are bounded, 
these problems are known to be uniquely solvable \cite[Thm. III.5.1]{LSU}.
Therefore we can use existence theorems for smoother solutions to establish higher regularity of $u$ and $v$.

Theorem IV.5.3 of \cite{LSU} asserts the existence of $C^{2+\al,1+\frac\al2}$ solutions, 
albeit under stronger smoothness assumptions on the initial datum than we can guarantee so far.
In order to nevertheless apply this theorem, let us, for $t_0>0$, $T>t>0$, introduce a smooth monotone function $\chi_{t_0,t,T}\col[t_0,t_0+T]\to \R$ satisfying $\chi(t_0)=0$ and $\chi\equiv 1$ on $[t_0+t,t_0+T]$ as well as $\norm[C^1(t_0,t_0+T)]{\chi_{t_0,t,T}}\leq 1+\frac2t$.

\begin{proposition}
 \label{prop:evsmooth}
 Let $n=3$ and assume that $\kappa<\kappa_0$ with $\kappa_0$ from Proposition \ref{prop:wirdmalklein}.
 Then there are $T^*>0$ and $\al>0$ such that $u,v\in C^{2+\al,1+\frac\al2}_{loc}(\Ombar\times[T^*,\infty))$.

 Moreover, there exists $C>0$ such that for all $t>T^*$
\[
 \norm[{C^{2+\al,1+\frac\al2}(\Ombar\times[t,t+1])}]{u}+\norm[{C^{2+\al,1+\frac\al2}(\Ombar\times[t,t+1])}]{v}\leq C.
\]
\end{proposition}
\begin{proof}
Let $T_*$ be as in Proposition \ref{prop:uepsfinallybounded} 
and $T_0>T_*+1$.
Let $\chi=\chi_{T_0,\frac12,2}$ as defined above and observe that $(\chi v)(T_0)=0$, $\delny (\chi v)\amrand=0$ and $\vtilde:=\chi v$ satisfies
\begin{equation}
 \label{eq:eq_higherregularityforv}
 \vtilde_t-\Lap \vtilde = \chi_t v+\chi u-\chi v\qquad\qquad\text{on }(T_0,T_0+2),
\end{equation}
a parabolic PDE with smooth coefficients and H\"older continuous right-hand side (due to Lemma \ref{lem:holdercts}).
Theorem IV.5.3 of \cite{LSU} in conjunction with the above-mentioned uniqueness property makes $\chi v$ an element of $C^{2+\al,1+\frac\al2}(\Ombar\times[T_0,T_0+2])$ and therefore $v\in C^{2+\al,1+\frac\al2}(\Ombar\times[T_0+\frac12,T_0+2])$, where, according to the aforementioned theorem, its norm can be estimated by the $C^{\al,\frac\al2}$-norm of the right-hand-side in \eqref{eq:eq_higherregularityforv} and therefore independently of $T_0>T_*+1$, cf. Lemma \ref{lem:holdercts}.

For an analogous procedure concerning $u$ let $\chi=\chi_{T_0+\frac12,\frac12,\frac32}$ and consider $\utilde=\chi u$, satisfying $\utilde(T_0+\frac12)=0$, $\delny \utilde\amrand=0$ and solving
\[
 \utilde_t-\Lap \utilde - \na \utilde\na v - \utilde\Lap v = \chi_t u +\chi (\kappa u-\my u^2),
\]
where the coefficients are H\"older continuous as well as the right hand side and, by the same argument as before, \cite[Thm. IV.5.3]{LSU} asserts $u\in C^{2+\al,1+\frac\al2}(\Ombar\times[T_0+1,T_0+2])$ with a $T_0$-independent estimate on the norm.
The claim follows upon the choice $T^*=T_0+1$ and due to the independence of the H\"older norm of $T_0$.
\end{proof}

After these preparations, the proof of our main result consists in nothing more than collating the right statements:
\begin{proof}[Proof of Theorem \ref{thm:eins}]
 Existence of a solution is given by Proposition \ref{prop:conv} in combination with Lemma \ref{lem:limesissoln}, eventual smoothness and bounds on the H\"older norms by Proposition \ref{prop:evsmooth}.
\end{proof}
\newcommand{\thetaa}{\vartheta}

\section{Asymptotic behaviour}
\label{sec:asymptotics}
Now that existence and smoothness of $(u,v)$ have been ensured, let us concentrate on the long time behaviour of solutions.

\subsection{The case $\kappa\leq0$. Proof of Theorem \ref{thm:kappazero}}
\label{sec:kappa_zero}
\begin{proof}[Proof of Theorem \ref{thm:kappazero}]
Let $\set{(u_{\eps_j},v_{\eps_j})}_{j\nat}$ be a sequence of solutions to \eqref{eq:epsprob} approaching $(u,v)$ in the sense of Proposition \ref{prop:conv}.
Let $\thetaa>0$. 

%
%
%
From Proposition \ref{prop:uepsfinallybounded} we can infer $\delta_0>0$ such that $K(\delta)$ from Proposition \ref{prop:uepsfinallybounded} satisfies $K(\delta)<\frac{\thetaa}3$ for any $\delta\in[0,\delta_0)$.

Now apply Lemma \ref{lem:nullstelle} with $\ny\in(0,\ny_0]$ so small that $\sqrt{\frac{4}{\Ceins(1+\frac1{4\ny})}}<\delta_0$ and choose $\kappatilde>0$ and $\eta\in(0,4]$ as provided therupon. In particular, this implies $\delta_\ny(\kappahat)\leq\delta_0$ 
for any $\kappahat\in(0,\kappatilde)$.

Let $\kappahat\in(0,\kappatilde)$ and let $T_0=T_0(\my,0,\kappahat)$ be as in Proposition \ref{prop:ODIulzvwev}.
%
%
As $\kappa\leq 0<\kappahat$, Proposition \ref{prop:uepsfinallybounded} implies that there is $T>0$ such that, independent of $j\nat$,
\begin{equation}
 \norm[\Liom]{u_{\eps_j}(t)}+\norm[W^{1,\infty}(\Om)]{v_{\eps_j}(t)}\leq 2K(\delta_\ny(\kappahat))+C e^{-(t-T)} \qquad\text{for all }t>T,
\end{equation}
where $C$ is a constant depending on the norm of the initial data $(u_0,v_0)$. 

Choose $T_\thetaa>T$ in such a way that $Ce^{-(T_\thetaa-T)}<\frac\thetaa3$ and that $u, v$ are continuous on $[T_\thetaa,\infty)$ by Theorem \ref{thm:eins}.

Our choice of $\delta_0$ thus shows that, independent of $j\nat$, 
\begin{equation*}
 \norm[\Liom]{u_{\eps_j}(t)}+\norm[W^{1,\infty}(\Om)]{v_{\eps_j}(t)}\leq 2\frac{\thetaa}3+\frac\thetaa3=\thetaa \qquad\text{for all }t>T_\thetaa.
\end{equation*}
Almost everywhere convergence of $(u_{\eps_j},v_{\eps_j})\to (u,v)$ (as stated by Proposition \ref{prop:conv} in \eqref{conv:uae}, \eqref{conv:vae}) and continuity of $u$ and $v$ hence imply that 
\begin{equation}
\label{eq:smallbound}
 \norm[\Liom]{u(t)}+\norm[\Liom]{v(t)}\leq \thetaa \qquad \text{for all }t>T_\thetaa.
\end{equation}

In conclusion, 
\[
 (u(t),v(t))\to 0 \qquad  \text{ as }t\to \infty
\]
in the sense of uniform convergence on $\Omega$.
\end{proof}

\subsection{Asymptotics for positive $\kappa$. Proof of Theorem \ref{thm:kappapositive}}
\begin{proof}[Proof of Theorem \ref{thm:kappapositive}]
 Under the condition of $\kappa$ being sufficiently small, 
 Theorem \ref{thm:eins} shows that the solutions constructed above enter some bounded set $B_{\my,\kappa}\sub (C^{2+\al}(\Ombar))^2$, where $\al>0$ is chosen as in Proposition \ref{prop:evsmooth}. 

 As to the statement about the diameter of $B_{\my,\kappa}$ in $\Liom\times W^{1,\infty}(\Om)$ as $\kappa\to 0$, we can proceed almost exactly as in the proof of Theorem \ref{thm:kappazero}: 
 Let $\thetaa>0$. From Proposition \ref{prop:uepsfinallybounded} we can infer $\delta_0>0$ such that $K(\delta)$ from Proposition \ref{prop:uepsfinallybounded} satisfies $K(\delta)<\frac{\thetaa}3$ for any $\delta\in[0,\delta_0)$. 
 The application of Lemma \ref{lem:nullstelle} with $\ny\in(0,\ny_0]$ satisfying $\sqrt{\frac{4}{\Ceins(1+\frac1{4\ny})}}<\delta_0$ provides $\eta\in(0,4]$ and $\kappatilde>0$. 
 Let $\kappahat\in(0,\kappatilde)$. \\
 We will prove that $\diam B_{\my,\kappa}\leq 2\thetaa$ if $\kappa<\kappahat$.\\
 Assume that $\kappa<\kappahat$ and let $T_0=T_0(\my,\kappa,\kappahat)$ be as in Proposition \ref{prop:ODIulzvwev}.
 As $\kappa<\kappahat$, Proposition \ref{prop:uepsfinallybounded} implies that there is $T>0$ such that, independent of $j\nat$,
 \begin{equation}
  \norm[\Liom]{u_{\eps_j}(t)}+\norm[W^{1,\infty}(\Om)]{v_{\eps_j}(t)}\leq 2K(\delta_\ny(\kappahat))+C e^{-(t-T)} \qquad\text{for all }t>T,
 \end{equation}
 where $C$ is a constant depending on the norm of the initial data $(u_0,v_0)$. 

 Choose $T_\thetaa>T$ in such a way that $Ce^{-(T_\thetaa-T)}<\frac\thetaa3$ and that $u, v$ are continuously differentiable on $[T_\thetaa,\infty)$ by Theorem \ref{thm:eins}.

 Our choice of $\delta_0$ thus shows that, independent of $j\nat$, 
\begin{equation*}
 \norm[\Liom]{u_{\eps_j}(t)}+\norm[W^{1,\infty}(\Om)]{v_{\eps_j}(t)}\leq 2\frac{\thetaa}3+\frac\thetaa3=\thetaa \qquad\text{for all }t>T_\thetaa.
\end{equation*}
We make use of the almost everywhere convergence of $(u_{\eps_j},v_{\eps_j})\to (u,v)$ (as stated by Proposition \ref{prop:conv} in \eqref{conv:uae}, \eqref{conv:vae}) and the fact that $\na v_{\eps_j}$ is essentially bounded by some constant $\Ctilde$ on $\Om\times[T_\thetaa,\infty)$ uniformly in $j$, which allows us to extract a $L^\infty$-weak$^*$-convergent subsequence leading to $\norm[\Liom]{\na v}\leq \Ctilde$. 

Together with the continuity of $u$, $v$ and $\na v$ these convergence results hence imply that 
\begin{equation*}
 \norm[\Liom]{u(t)}+\norm[W^{1,\infty}(\Om)]{v(t)}\leq \thetaa \qquad \text{for all }t>T_\thetaa.
\end{equation*}
In terms of $B_{\my,\kappa}$ this means 
\[
 B_{\my,\kappa}\subset B_\thetaa^{\Liom\times W^{1,\infty}(\Om)}(0)
\]
and hence $\diam (B_{\my,\kappa}) \leq 2\thetaa$ for sufficiently small $\kappa>0$.
\end{proof}

\end{document}